\documentclass[a4paper,12pt]{amsart}
\usepackage{amsmath, amsthm,amssymb, latexsym}
\usepackage{tikz,hyperref,pgf,pgflibraryarrows,enumerate, amscd, amssymb}
\usepackage[all]{xy}
\usepackage[utf8]{inputenc}
\usepackage[T1]{fontenc}
\usepackage{url}

\newtheorem{thm}{Theorem}[section]
\newtheorem{cor}[thm]{Corollary}

\newtheorem{prop}[thm]{Proposition}

\theoremstyle{definition}
\newtheorem{dfn}[thm]{Definition}

\theoremstyle{remark}
\newtheorem{rmk}[thm]{Remark}

\newcommand{\Z}{\mathbb{Z}}

\newcommand{\J}{\mathbb J}

\newcommand{\ZZ}{\Z}
\newcommand{\RR}{\mathbb{R}}
\newcommand{\TT}{\mathbb{T}}

\newcommand{\CC}{\mathbb{C}}
\newcommand{\PP}{\text{Pic}}
\newcommand{\CPP}{\text{CPic}}

\newcommand{\Pic}{\operatorname{Pic}}
\newcommand{\Aut}{\operatorname{Aut}}
\newcommand{\Gin}{\operatorname{Gin}}
\newcommand{\Br}{\operatorname{Br}}
\newcommand{\lspan}{\operatorname{span}}
\def\clsp{\operatorname{\overline{span}}}

\title{Quantum Heisenberg manifolds as twisted groupoid $C^*$-algebras}
\author[Kang]{Sooran Kang}
\address{Sooran Kang\\ Department of Mathematics and Statistics\\University of Otago\\PO Box 56\\Dunedin 9054\\New Zealand}
\email{sooran@maths.otago.ac.nz}

\author[Kumjian]{Alex Kumjian}
\address{Alex Kumjian\\ Department of Mathematics\\University of Nevada\\Reno NV 89557\\USA}
\email{alex@unr.edu}

\author[Packer]{Judith Packer}
\address{Judith Packer\\ Department of Mathematics\\ University of Colorado\\ Campus Box 395\\ Boulder CO 80309-0395\\ USA}
\email{packer@euclid.colorado.edu}

\subjclass[2010]{46L05 (primary),  46L55 (secondary)}
\keywords{Quantum Heisenberg manifolds; Twisted groupoid $C^*$-algebras; Principal $G$-bundles}
\begin{document}

\maketitle

\begin{abstract}
The quantum Heisenberg manifolds are noncommutive manifolds constructed by M. Rieffel as strict deformation quantizations of Heisenberg manifolds and have been studied by various authors. Rieffel constructed the quantum Heisenberg manifolds as the generalized fixed-point algebras of certain crossed product $C^*$-algebras, and they also can be realized as crossed products of $C(\TT^2)$ by Hilbert $C^*$-bimodules in the sense of Abadie et al.
In this paper, we describe how the quantum Heisenberg manifolds can also be realized as  twisted groupoid $C^*$-algebras.

\end{abstract}

\section{Introduction}

Let $X$ be a compact Hausdorff space and $\sigma$  a covering map from $X$ to itself.  The Renault-Deaconu groupoid $\Gamma$ associated with $(X,\sigma)$ was  introduced by Renault in \cite{Renextra} and studied by Deaconu in greater generality in \cite{D1}.  The twist groupoid $\Lambda$ over $\Gamma, $ in the sense of the second author \cite{Ku1,Ku2}, and its associated groupoid $C^*$-algebra $C^*(\Gamma; \Lambda)$ were studied in \cite{DKM} for more general $X$ and $\sigma$, in particular when $X$ is locally compact Hausdorff and $\sigma$ is a local homeomorphism. The twist groupoid $\Lambda$ was defined as an extension of $\Gamma$ by the groupoid $X\times \TT,$ with the notion of the twist that determines  $\Lambda$ generalizing circle-valued 2-cocycles over groups. One of the main results in \cite{DKM} is that for each twist $\Lambda$ over $\Gamma$, the groupoid $C^*$-algebra $C^*(\Gamma,\Lambda)$ is naturally isomorphic to a Cuntz-Pimsner algebra $\mathcal{O}_{\mathcal{H}}$, where $\mathcal{H}$ is a Hilbert $C^*$-correspondence. In particular, it was shown in \cite{DKM} that the $C^*$-correspondence $\mathcal{H}$ was given by $L_T\otimes \ell^2(\sigma),$ where $L_T$ is the space of continuous sections of the complex line bundle arising from the circle bundle associated to the twist $\Lambda,$ and $\ell^2(\sigma)$ is a $C^*$-correspondence associated to the local homeomorphism $\sigma:X\to X$.

The Picard group of a $C^*$-algebra was introduced by Brown, Green and Rieffel in \cite{BGR}, and consists of the isomorphism classes of full Hilbert $C^*$-bimodules (i.e. the imprimitivity bimodules over $C^*$-algebras in the sense of Rieffel \cite{Rie1}), equipped with the balanced tensor product. When the $C^*$-algebra is commutative, the collection of symmetric Hilbert $C^*$-bimodules in the sense of Definition \ref{symm-module} determine the classical Picard group, and have elements that can be regarded as the space of continuous sections of complex line bundles over the spectrum of the $C^*$-algebra in question. Thus, the elements of the classical Picard group are related to  twist groupoids $\Lambda$ of a Renault-Deaconu groupoid $\Gamma$ (see Section \ref{section_twist}) by results from \cite{DKM}.

The aim of this paper is to use the notion of a twisted groupoid $C^*$-algebra to study the quantum Heisenberg manifolds of M. Rieffel.
Quantum Heisenberg manifolds (henceforth abbreviated by QHMs) were introduced by Rieffel in \cite{Rie3} as strict deformation quantizations of Heisenberg manifolds. They are the fibers of a continuous field of $C^*$-algebras $\{D^{c,\hslash}_{\mu,\nu}\}$ where $\mu,\nu\in\RR$, $c$ is a positive integer and $\hslash$ is the deformation parameter representing the Planck constant. Rieffel's strict deformation quantization provides a way to construct noncommutative differential manifolds, and the QHMs form another interesting family of noncommutative manifolds that moves beyond noncommutative tori. Using the theory of Morita equivalence of QHMs developed by Abadie in \cite{Ab1,Ab3,Ab2}, Kang studied Yang-Mills theory on QHMs in \cite{Kang1}, and Lee found a minimizing set of the Yang-Mills functional in \cite{Lee}, using the same module developed by Kang in \cite{Kang1}. Chakraborty and Shinha studied the geometry of QHMs in \cite{Ch1, CS} and Gabriel studied cyclic cohomology and index pairings of QHMs in \cite{G}. Recently it was shown in \cite{Ch2} that in the case of QHMs, the Yang-Mills functional coming from spectral triples are the same as the Yang-Mills functional defined by Connes and Rieffel in \cite{CR}.

It was shown in \cite{AEE} that every QHM can be realized as a crossed product by a Hilbert $C^*$-bimodule. Such a crossed product is in particular a Cuntz-Pimsner algebra (see \cite{Kat2}). In fact, the QHM is generated by the fixed point algebra under the dual action and the first spectral subspace denoted by $D_1$. i.e. $D^{c,\hslash}_{\mu\nu}\cong C(\TT^2)\rtimes_{D_1}\ZZ$.
Also \cite[Corollary~5.5]{Ab4} implies that $D^{c,\hslash}_{\mu\nu}$ is strongly Morita equivalent to $(C(\TT^2)\otimes \mathcal{K})\rtimes_\beta \mathbb{Z}$ for some $\beta\in \Aut(C(\TT^2)\otimes \mathcal{K})$, where $\mathcal{K}$ is the algebra of compact operators on an infinite Hilbert space.  Hence the QHM gives an element of the equivariant Brauer group $\Br_{\mathbb Z}(\TT^2)$ in the sense of \cite{CKRW,PRW}.
Moreover, the isomorphism classes of QHMs have been studied using the Picard group in \cite{AE}; in particular, the Picard group of $\TT^2$ played an important role in determining the isomorphism classes of QHMs in \cite{AE}.

Upon expressing the QHM, $D^{c,\hslash}_{\mu\nu}$, as a crossed product by a Hilbert $C^*$-bimodule $C(\TT^2)\rtimes_{D_1}\ZZ,$ one can obtain a groupoid $\Gamma$ and a twist $\Lambda$ that have associated to them the circle bundle $T$ over the unit space of $\Lambda$.  It is then possible to construct the space of continuous sections of the complex line bundle associated to $T$. This in turn becomes a symmetric $C^*$-bimodule over a commutative $C^*$-algebra isomorphic to $C(\TT^2),$ which can be identified with the symmetric part of the first spectral subspace $D_1$ of QHMs. Intertwining these concepts, in this paper we will describe how QHMs can be realized as twisted groupoid $C^*$-algebras and give a specific construction of the twist groupoid $\Lambda$ from a transformation groupoid by means of transition functions of the circle bundle $T$.  Indeed our main theorem (Theorem~\ref{thm_main}) holds in greater generality for certain $C^{\ast}$-algebras constructed as a particular crossed product of a unital commutative $C^{\ast}$-algebra $A$ by a full Hilbert $C^*$-bimodule $M$ over $A$ twisted by an automorphism $\alpha$ of $A$.

Since the algebra structure of QHMs is quite different from that of noncommutative tori, not being defined by generators and relations, it is quite useful to have another way of realizing QHMs.
  We believe that by illustrating that the QHMs and other generalized fixed point algebras can be constructed in a variety of different fashions, it may lead to further progress in the study of noncommutative geometry on some dense subalagebras generally, and in particular in the study of the Yang-Mills theory of the QHMs.

We organize our paper as follows.
After reviewing some preliminary material in Section 2, we discuss the classical Picard group and review the description of the QHM as a crossed product by a Hilbert $C^*$-bimodule in Section 3. Moreover, we identify the symmetric part of the first spectral subspace $D_1$ of the QHM   with the symmetric Hilbert $C(\TT^2)$-bimodule $M^c$ in Proposition~\ref{prop:D1-module}, where $c$ is a positive integer.
In Section 4, we prove our main theorem and describe QHMs as certain twisted groupoid $C^*$-algebras by showing that $M^{-c}$ is isomorphic to the space of continuous sections of the complex line bundle associated to the Heisenberg manifold $N_c$. In Section 5, we obtain a pairing from a transformation groupoid $\Gamma$ and a principal $G$-bundle, for some compact abelian group $G,$ to form the associated `twist' groupoid $\Lambda$, where $\Lambda$ is a principal $G$-bundle over $\Gamma$.  In the case where $G=\mathbb T,$ these are the twist bundles of \cite{DKM}.
Of course one obtains the $C^{\ast}$-algebras only in the cases where $G=\mathbb T,$ but we thought these more general groupoids might also be of interest.

\subsection*{Acknowledgements}
This research project was originally motivated by a question posed by Olivier Gabriel, who asked if the quantum Heisenberg manifold could be expressed as a twisted groupoid $C^*$-algebra. This question was communicated to us by Jean Renault.

The second and third authors would like to thank the first author and Aidan Sims for hosting them at the University of Wollongong where this research collaboration began, and also would like to thank the faculty members at School of Mathematics and Applied Statistics of the University of Wollongong for their hospitality.

\section{Preliminaries}\label{sec:prem}
Most of the definitions that we state in this section are from \cite{Kat1}.
\begin{dfn}\cite[Definition 1.1]{BMS}\label{left-module}
Let $A$ be a $C^*$-algebra. A \emph{right Hilbert $A$-module} $M$ is a Banach space with a right action of the $C^*$-algebra $A$ and an $A$-valued inner product $\langle\cdot,\cdot\rangle_R$ such that for $\xi,\eta\in M$ and $a\in A$,
\begin{enumerate}
\item[$(a)$] $\langle\xi,\xi\rangle_R\ge 0$ and $\lVert\xi\rVert=\lVert\langle\xi,\xi\rangle_R\rVert^{1/2}$,
\item[$(b)$] $\langle\xi,\xi\rangle_R=0$ implies $\xi=0$,
\item[$(c)$] $\langle\xi,\eta\rangle_R=(\langle\eta,\xi\rangle_R)^*$,
\item[$(d)$] $\langle\xi,\eta\cdot a\rangle_R=\langle\xi,\eta\rangle_R a$.
\end{enumerate}
\end{dfn}

Similarly we can define a \emph{left Hilbert $A$-module} $M$ as a Banach space with a left action of the $C^*$-algebra $A$ and an $A$-valued inner product
$\langle\cdot,\cdot\rangle_L$ that satisfies $(a),(b),(c)$ of Definition \ref{left-module} and
\begin{itemize}
\item[$(d)'$] $\langle a\cdot \xi,\eta\rangle_L=a\langle \xi,\eta\rangle_L$.
\end{itemize}
We say that a Banach space $M$ is a Hilbert $A$-$B$ bimodule if $M$ is both a left Hilbert $A$-module and a right Hilbert $B$-module and satisfies the following equation : for $\xi,\eta,\zeta\in M$,
\begin{equation}\label{eq:bimodule}
\langle \xi,\eta\rangle_L\cdot \zeta=\xi\cdot\langle \eta,\zeta\rangle_R.
\end{equation}
If $A=B$, then we call it a Hilbert $A$-bimodule.

\begin{dfn}\cite[Definition 1.2]{Kat1}
For a right Hilbert $A$-module $M$, we denote by $\mathcal{L}(M)$ the $C^*$-algebra of all adjointable operators on $M$. For $\xi,\eta\in M$, the operator $\theta_{\xi,\eta}\in\mathcal{L}(M)$ is defined by $\theta_{\xi,\eta}(\zeta)=\xi\langle\eta,\zeta\rangle_R$ for $\zeta\in M$. We define $\mathcal{K}(M)\subseteq\mathcal{L}(M)$ by
\[\mathcal{K}(M)=\overline{\text{span}}\{\theta_{\xi,\eta}:\xi,\eta\in M\}.\]
\end{dfn}
The set $\mathcal{K}(M)$ is a (two-sided and ${}^*$-closed) ideal of $\mathcal{L}(M)$, which is *-invariant. We often call $\mathcal{K}(M)$ the algebra of compact operators on $M$.

\begin{dfn}\cite[Definition 1.3]{Kat1}
 We say that a right Hilbert $A$-module $M$ is a \emph{$C^*$-correspondence} over $A$ when a *-homomorphism $\phi_M:A\rightarrow\mathcal{L}(M)$ is given.
\end{dfn}
We refer to $\phi_M$ as the left action of $A$ on the $C^*$-correspondence $M$. If $M$ is a Hilbert $A$-bimodule, then $M$ can be realized as a $C^*$-correspondence over $A$ with $\phi_M: A\rightarrow\mathcal{L}(M)$ given by $\phi_M(a)\xi= a\cdot \xi$.
\begin{dfn}\cite[Definition 1.4]{Kat2}
A $C^*$-correspondence $M$ over a $C^*$-algebra $A$ is called \emph{non-degenerate} if $\overline{\text{span}}\{\phi_M(a)\xi\in M: a\in A,\xi\in M\}=M$, and \emph{full} if $\overline{\text{span}}\{\langle\xi,\eta\rangle_R\in A:\xi,\eta\in M\}=A$.
\end{dfn}

\begin{dfn}\cite[Definition 3.2]{Kat1}
For a $C^*$-correspondence $M$ over $A$, we define an ideal $J_M$ of $A$ by
\[J_M=\phi^{-1}_M(\mathcal{K}(M))\cap(\text{ker}\phi_M)^{\bot}.\]
\end{dfn}

\begin{rmk}
Note that $J_M=\phi^{-1}_M(\mathcal{K}(M))$ when $\phi_M$ is injective.
According to \cite[Lemma~2.4]{Kat2}, if there exists an ideal $J$ of $A$ such that the restriction of $\phi_M$ to $J$ is an isomorphism onto $\mathcal{K}(M)$, then $J=J_M$.

Moreover, \cite[Lemma~3.4]{Kat2} shows that a Hilbert $A$-bimodule $M$ is the $C^*$-correspondence $M$ over $A$ with the condition $\phi_M(J_M)=\mathcal{K}(M))$, and it was referred to as a Hilbert $C^*$-bimodule $M$ over $A$ by Abadie et al. in \cite{AEE}. A full Hilbert $A$-bimodule $M$ (i.e. $\langle M,M\rangle_L=\langle M,M\rangle_R=A$) is a left $A$ - right $A$ imprimitivity bimodule $M$ in the sense of Rieffel given in \cite{Rie1}.
\end{rmk}

For two $C^*$-correspondences $X$ and $Y$ over a $C^*$-algebra $A$, one can define a tensor product of $X$ and $Y$ denoted by $X\otimes_A Y$ as in \cite[Definition 1.4]{Kat1}. By definition, we have
\[
X\otimes_A Y=\overline{\lspan}\{\xi\otimes\eta\mid \xi\in X, \eta\in Y\}
\]
and $(\xi\cdot a)\otimes \eta =\xi\otimes (\phi_Y(a)\cdot\eta)$ for $\xi\in X,\eta\in Y$ and $a\in A$.

\begin{dfn}\cite[Definition~2.1]{Kat1}\label{Def:Corr}
A \emph{representation} of a $C^*$-correspondence $M$ over $A$ on a $C^*$-algebra $B$ is a pair $(\pi,t)$ consisting of a *-homomorphism $\pi:A\rightarrow B$ and a linear map $t:M\rightarrow B$ satisfying, for $\xi,\eta\in M$ and $a\in A$,
\begin{enumerate}
\item[$(a)$] $t(\xi)^*t(\eta)=\pi(\langle\xi,\eta\rangle_R)$,
\item[$(b)$] $\pi(a)t(\xi)=t(\phi_M(a)\xi)$.
\end{enumerate}
\end{dfn}
For a representation $(\pi,t)$ on $B$, we denote by $C^*(\pi,t)$ the $C^*$-algebra generated by the images of $\pi$ and $t$ in $B$.

\begin{dfn}\cite[Definition 2.3]{Kat1}
For a representation $(\pi,t)$ of a $C^*$-correspondence $M$ on $B$, we define a *-homomorphism $\psi_t:\mathcal{K}(M)\rightarrow B$ by $\psi_t(\theta_{\xi,\eta})=t(\xi)t(\eta)^*\in B$ for $\xi,\eta\in M$.
\end{dfn}
Note that $\psi_t(\mathcal{K}(M))\subseteq C^*(\pi,t)$.

\begin{dfn}\cite[Definition~3.5]{Kat1}\label{Def:CP}
For a $C^*$-correspondence $M$ over a $C^*$-algebra $A$, the $C^*$-algebra $\mathcal{O}_M$ is defined by $\mathcal{O}_M=C^*(\pi_M,t_M)$, where $(\pi_M,t_M)$ is the universal representation of $M$ satisfying the condition
\begin{equation}\label{eq:CP_condition}
\pi_M(a)=\psi_t(\phi_M(a))\;\;\text{for all}\;\; a\in J_M.
\end{equation}
 We call $\mathcal{O}_M$ the \emph{Cuntz-Pimsner algebra associated with $M$}.
\end{dfn}

We now review crossed products of $C^*$-algebras by Hilbert $C^*$-bimodules as defined by Abadie, Eilers, and Exel.

\begin{dfn}\cite[Definition~2.1, Definition~2.4]{AEE}\label{def:AEE}
Let $A$ be a $C^*$-algebra and $M$ be a Hilbert $A$-bimodule. A \emph{covariant representation} of $(A,M)$ on $B(\mathbb{H})$ is a pair $(\kappa_A, \kappa_M)$ consisting of a $*$-homomorphism $\kappa_A:A\to B(\mathbb{H})$ and a linear map $\kappa_M:M\to B(\mathbb{H})$ satisfying the following: For $a\in A$ and $x,y\in M$,
\begin{enumerate}
\item[$(a)'$] $\kappa_M(a\cdot x)=\kappa_A(a)\kappa_M(x)$,
\item[$(b)'$] $\kappa_M(x\cdot a)=\kappa_M(x)\kappa_A(a)$,
\item[$(c)'$] $\kappa_A(\langle x,y\rangle_L)=\kappa_M(x)\kappa_M(y)^*$,
\item[$(d)'$] $\kappa_A(\langle x,y\rangle_R)=\kappa_M(x)^*\kappa_M(y)$.
\end{enumerate}
A \emph{crossed product} $A\rtimes_M\ZZ$ of $A$ by $M$ is the universal $C^*$-algebra generated by the images of covariant representations $(\kappa_A, \kappa_M)$ of $(A,M)$.
\end{dfn}

As discussed in \cite{Kat2}, conditions $(a)'$ and $(d)'$ of Definition~\ref{def:AEE} are equivalent to the conditions for representations of $M$ given in Definition~\ref{Def:Corr} as a $C^*$-correspondence with $\phi_M$ given by $\phi_M(a)\xi=a\cdot \xi$ for $a\in A$ and $\xi\in M$. Since $M$ is a Hilbert $A$-bimodule, \cite[Lemma~3.3]{Kat1} gives $J_M=\clsp\{\langle \xi,\eta\rangle_L\in A\mid \xi,\eta\in M\}$. Also the bimodule condition \eqref{eq:bimodule} gives $\phi_M(\langle \xi,\eta\rangle_L)=\theta_{\xi,\eta}$ for $\xi,\eta\in M$. Thus we see the condition $(c)'$ is equivalent to the condition \eqref{eq:CP_condition} in Definition~\ref{Def:CP}. Also the condition $(b)'$ follows from $(a)', (c)'$ and $(d)'$. Therefore, when the $C^*$-correspondence comes from a Hilbert $A$-bimodule $M$, $\mathcal{O}_M$ is isomorphic to the crossed product of $A$ by the Hilbert $C^*$-bimodule $M,\;A\rtimes_M\ZZ$ defined in \cite{AEE}. (Further details on this can be found in \cite{Kat1,Kat2}).

\section{The Classical Picard Groups}\label{sec:pic}

Let $A$ be a $C^*$-algebra. The isomorphism classes of full Hilbert $A$-bimodules form a group under tensor product as defined in \cite[Section 5.9]{Rie1} which we call the \emph{Picard group} of $A,$ and denote by $\PP(A),$ as in \cite{BGR}.
By an isomorphism of left (respectively, right) Hilbert $A$-modules, we mean an isomorphism of left (respectively right) $A$-modules that preserves the left (respectively, right) inner product. An isomorphism of Hilbert $A$-bimodules is an isomorphism of both left and right Hilbert $A$-modules.
Let $A$ be a $C^*$-algebra and let $u$ be a unitary element of the multiplier algebra of $A$.
According to \cite[Proposition 3.1]{BGR}, there is an anti-homomorphism from $\Aut(A)$ to $\Pic(A)$ such that the sequence
\[\label{exact-seq}
1 \to \Gin(A) \to \Aut(A) \to \Pic(A)
\]
is exact, where $\Gin(A)$ is the group of all generalized inner automorphisms $Ad_u$ of $A$ given by $Ad_u(a)=uau^{-1}$ for $a\in A$.
  Thus an automorphism $\theta$ of $A$ is mapped to a Hilbert $C^*$-bimodule $A_\theta$ that is determined by $\theta$. In fact, $A_\theta$ is the vector space $A$ with the following actions and $A$-valued inner products: for $\xi,\eta,a\in A$,
\[\begin{split}
a\cdot \xi=a\xi, \quad &\xi\cdot a=\xi\theta(a),\\
\langle \xi,\eta\rangle_L=\xi\eta^*, \quad &\langle \xi,\eta\rangle_R=\theta^{-1}(\xi^*\eta).
\end{split}
\]
 Note that we choose to follow the notation of Abadie and Exel given in \cite{AE} rather that given in Brown, Green, and Rieffel's version from \cite{BGR}; this implies that we have $A_\alpha\otimes A_\beta\cong A_{\alpha\beta}$ for $\alpha, \beta\in \Aut(A)$.

For a full Hilbert $A$-bimodule $M$ with $\langle M,M\rangle_L=\langle M,M\rangle_R=A$, and for $\theta\in\Aut(A)$, define $M_\theta$ to be $M\otimes_A A_\theta$. Then as shown in \cite{AE}, the map $\xi\otimes a \mapsto \xi\cdot a$ identifies $M_\theta$ with the vector space $M$ that carries the following actions and inner-products: for $\xi,\eta\in M$ and $a\in A$,
\[\begin{split}
a\cdot_{M_\theta} \xi= a\cdot_M \xi,\quad & \xi\cdot_{M_\theta} a=\xi\cdot_M \theta(a),\\
\langle \xi,\eta\rangle^{M_\theta}_L=\langle \xi,\eta\rangle^M_L,\quad &\langle \xi,\eta\rangle^{M_\theta}_R=\theta^{-1}(\langle \xi,\eta\rangle^M_R).
\end{split}
\]
 The general formulas for $M_\theta$ are given in \cite[p.413]{AE} when $\theta$ is a partial automorphism of $A$. When the actions and inner-products are clear from context, we drop the subscripts and superscripts $M_\theta$ from the formulas.

  Similarly, for $\alpha\in\Aut(A)$ and a Hilbert $A$-bimodule $M$, the map $a\otimes \xi\otimes b \mapsto a\cdot \xi\cdot b$ identifies the Hilbert $A$-bimodule $A_\alpha\otimes_A M\otimes_A A_{\alpha^{-1}}$ with the vector space $M$ that carries the following operations: for $\xi,\eta\in M$ and $a\in A$,
\[\begin{split}
a\cdot \xi =\alpha^{-1}(a)\xi,\quad & \xi\cdot a=\xi\alpha^{-1}(a),\\
\langle \xi,\eta\rangle_L=\alpha(\langle \xi,\eta\rangle^M_L),\quad & \langle \xi,\eta\rangle_R=\alpha(\langle \xi,\eta\rangle^M_R).
\end{split}\]

When $A=C(X)$ for a compact Hausdorff space $X$, $\PP(A)$ contains isomorphism classes of what are termed \emph{symmetric} (see Definition \ref{symm-module}) Hilbert $A$-bimodules in \cite{AE}. A symmetric Hilbert $A$-bimodule may be regarded as the space of sections of a complex line bundle over $X$. (Readers are referred to Appendix (A) in \cite{Rae1} for further details on this). We call this subgroup of $\PP(A)$ the \emph{classical Picard group} of $C(X),$ and it will be denoted by $\CPP(C(X))$.
In the case where $X=\TT^2$, we have $\Pic(C(\TT^2))\cong H^2(\TT^2,\ZZ)\cong \ZZ$ by \cite[\S2]{PR}.

\begin{dfn}\cite[Definition~1.5]{AE}\label{symm-module}
Let $A$ be a commutative $C^*$-algebra. A Hilbert $A$-bimodule $M$ is said to be \emph{symmetric} if $a\cdot\xi=\xi\cdot a$ for $\xi\in M$ and $a\in A$.
If $M$ is a Hilbert $A$-bimodule over $A$, the \emph{symmetrization} of $M$ is the symmetric Hilbert $A$-bimodule $M^s$ whose underlying vector space is $M$ with its given left Hilbert-module structure, and right module structure defined by:
\[\xi\cdot a=a\xi,\;\;\langle \xi,\eta\rangle^{M^s}_R=\langle \xi,\eta\rangle^M_L,\]
for $a\in A, \xi,\eta\in M^s$.
\end{dfn}

\begin{rmk}
The commutativity of $A$ guarantees the compatibility of the left and right actions for $M^s$.
Also it is straightforward to show that $\langle \xi,\eta\rangle^{M_s}_L\cdot \zeta = \xi\cdot \langle \eta,\zeta\rangle^{M_s}_R$ for all $\xi,\eta,\zeta\in M^s$.
As mentioned in Remark 1.6 in \cite{AE}, the bimodule $M^s$ is, up to isomorphism, the only symmetric Hilbert $A$-bimodule that is isomorphic to $M$ as a left Hilbert module.
\end{rmk}

The following proposition was implicitly shown in \cite{AE} for any Hilbert $C^*$-bimodule. We state it here explicitly for the case of full $C^*$-bimodules for later reference.
\begin{prop}\label{prop:symm}
Let $A$ be an commutative $C^*$-algebra. Let $M$ be a full Hilbert $A$-bimodule. i.e. $\langle M,M\rangle_L=\langle M,M\rangle_R=A$.
Then $M$ is isomorphic to $M^s\otimes_A A_{\alpha}$ for some $\alpha\in\text{Aut}(A)$.
\end{prop}
\begin{proof}
Since $M$ is a full Hilbert $A$-bimodule, $\langle M,M\rangle_R=A$ and $\langle M^s,M^s\rangle_R=\langle M,M\rangle_L=A$.
Then the proof of \cite[Proposition 1.1]{AE} gives an automorphism $\alpha:A\to A$ such that $(M^s)\otimes_A A_\alpha$ is isomorphic to $M$.
\end{proof}

The following theorem is due to Abadie and Exel.

\begin{thm}\cite[Theorem 1.12]{AE}
Let $A$ be a commutative $C^*$-algebra. Then $\CPP(A)$ is a  normal subgroup of $\PP(A)$ and
\[\PP(A)=\CPP(A)\rtimes\text{Aut}(A),\]
where the action of $\text{Aut}(A)$ on $\CPP(A)$ is given by conjugation, $\theta\cdot M=A_{\theta}\otimes_A M\otimes_A A_{\theta^{-1}}$.
\end{thm}

 Thus it follows that every full Hilbert $A$-bimodule is isomorphic to a tensor product of a symmetric full Hilbert $A$-bimodule and a Hilbert $A$-bimodule of the form $A_\theta$ where $\theta$ is an automorphism of $A$.

Now we describe a QHM as a crossed product by a full Hilbert $C(\TT^2)$-bimodule, and identify the symmetric part of the first spectral space of the QHM as follows.
According to the original construction by Rieffel \cite{Rie3}, the QHMs $ \{D^{c,\hslash}_{\mu\nu}\}_{\hslash\in \RR} $ are generalized fixed point algebras of certain crossed product $C^*$-algebras under proper actions.
 In particular, the algebra $D^{c,\hslash}_{\mu\nu}$ can be described as follows. (In the sequel, we follow Abadie's notation as given in \cite{Ab1}).
Since our discussion will not depend on $\hslash$, and since $D^{c,\hslash}_{\mu\nu}\cong D^{c,1}_{\mu'\nu'}$ for appropriate choices of $\mu'$ and $\nu',$ in what follows, we let $\hslash=1$ and drop the superscript.

 Let $\lambda$ and $\sigma$ be the commuting actions of $\Z$ on $\RR\times \TT$ arising from the homeomorphisms
\[\lambda_p(x,y)=(x+2 p\mu,y+2 p\nu)\;\;\text{and}\;\;\sigma_p(x,y)=(x-p,y),\]
where $\mu,\nu\in\RR$ and $p\in\Z$. Then, forming the crossed product $C^*$-algebra $C_b(\RR\times\TT)\rtimes_{\lambda}\Z$
in the usual way, we define an action $\rho$ of $\Z$ on $C_b(\RR\times\TT)\rtimes_{\lambda}\Z$ by
\[(\rho_k\Phi)(x,y,p)=\overline{e}(ckp(y- p\nu))\Phi(x+k,y,p),\]
where $e(x)=exp(2\pi ix)$ and $c$ is a positive integer. The generalized fixed point algebra $D^{c}_{\mu\nu}$ of $C_b(\RR\times\TT)\rtimes_{\lambda}\Z$
by the action $\rho$ is defined to be the closure in the multiplier algebra of $C_b(\RR\times\TT)\rtimes_{\lambda}\Z$ of the *-subalgebra
consisting of functions $\Phi\in C_c(\RR\times\TT\times\Z)$ that satisfy $\rho_k(\Phi)=\Phi$ for $k\in\Z$.
By definition, we have
\[D^{c}_{\mu\nu}=\overline{\text{span}}\{\Phi\in C_c(\RR\times\TT\times\Z)\mid\overline{e}(ckp(y- p\nu))\Phi(x+k,y,p)=\Phi(x,y,p)\;\;\text{for}\;\;k\in\Z\}.\]
There is a natural dual action $\gamma$ of $\TT$ on $D^{c}_{\mu\nu}$ given by
\begin{equation}\label{eq:action_gamma}
(\gamma_z\Phi)(x,y,p)=z^p\Phi(x,y,p)=e(pz)\Phi(x,y,p).
\end{equation}
Then as in \cite[Example~3.3]{AEE}, the $n$-th spectral subspace $D_n$ of $D^{c}_{\mu\nu}$ is given by:
\[D_n =\{f\delta_n|f\in C_b(\RR\times\TT), f(x+1,y)=e(cn(y- n\nu))f(x,y)\}.\]
It is obvious that the fixed point algebra $D_0$ is isomorphic to $C(\TT^2)$ and it is routine to check that the first spectral subspace $D_1$ has the structure of a $C(\TT^2)$-bimodule, with the following formulas:
For $\phi\in C(\TT^2)$ and $g_1,g_2\in D_1$,
\[(\phi\cdot g_1)(x,y)=\phi(x,y)g_1(x,y),\;\;(g_1\cdot\phi)(x,y)=g_1(x,y)\lambda(\phi)(x,y),\]
\[\langle g_1,g_2\rangle_L(x,y)=g_1(x,y)\overline{g_2}(x,y),\;\;\langle g_1,g_2\rangle_R(x,y)=(\lambda^{-1}(\overline{g_1}g_2))(x,y),\]
where $\lambda$ is the action of $\Z$ on $\RR\times\TT$ given by $\lambda(x,y)=(x+2\mu,y+2\nu)$.
According to \cite[Example~3.3]{AEE}, $D_1D^*_1=D^*_1 D_1=C(\TT^2),$ which implies that $D_1$ is a full $C(\TT^2)$-bimodule.

The following results due to Abadie, Exel and Eiler are found in \cite{AE, AEE},  and we state them here in one proposition for later reference.

\begin{prop}\cite{AE,AEE}\label{prop:D1-module}
Let $c$, $\mu$, $\nu$, $D^c_{\mu\nu}$ and $D_1$ be as above. Let $\alpha$ be the homeomorphism on $\TT^2$ defined by $\alpha(x,y)=(x+2\mu, y+2\nu)$ and $\alpha'\in\Aut(C(\TT^2))$ induced from $\alpha$.
Let
\[M^c=\{f\in C_b(\RR\times \TT)\mid f(x+1,y)=e(cy)f(x,y)\}.\]
Then
\begin{itemize}
\item[$(a)$] $M^c$ is a symmetric $C(\TT^2)$-bimodule,
\item[$(b)$] $D^c_{\mu\nu}\cong C(\TT^2)\rtimes_{D_1}\ZZ$, and
\item[$(c)$] $D_1\cong M^c\otimes(C(\TT^2))_{\alpha'}$.
\end{itemize}
\end{prop}

\begin{proof}
It is routine to check that $M^c$ is a symmetric $C(\TT^2)$-bimodule with pointwise actions and canonical inner products.
Also, it is shown in \cite[Example~3.3]{AEE} that the circle action $\gamma$ in \eqref{eq:action_gamma} is semi-saturated, thus \cite[Theorem~3.1]{AEE} gives $D^c_{\mu\nu}\cong C(\TT^2)\rtimes_{D_1}\Z$.

To see $(c),$ let $\Phi: g\mapsto \widetilde{g}$ be the map defined in \cite[\S2]{AE} by $\widetilde{g}(x,y)=g(x,y+\nu)$ for $(x,y)\in \RR\times \TT$. Then $\Phi$ provides an isomorphism of Hilbert $C(\TT^2)$-bimodules between $D_1$ and $C(\TT^2)_{\delta'}\otimes M^c\otimes C(\TT^2)_{\tau'}$, where $\delta',\tau'\in\Aut(C(\TT^2))$ are induced from the maps $\delta, \tau :\TT^2\to \TT^2$ defined by $\delta(x,y)=(x,y+\nu)$ and $\tau(x,y)=(x+2\mu,y+\nu)$.  Since $M^c$ is symmetric and $\delta'$ is homotopic to the identity,  \cite[Proposition~1.15]{AE} implies
 \[\begin{split}
 D_1 &\cong C(\TT^2)_{\delta'}\otimes M^c\otimes C(\TT^2)_{\tau'}\cong M^c\otimes C(\TT^2)_{\delta'}\otimes C(\TT^2)_{\tau'}\\
 &\cong M^c\otimes C(\TT^2)_{(\tau\delta)'}\cong M^c\otimes C(\TT^2)_{\alpha'},
 \end{split}\]
  where $\alpha'\in\Aut(C(\TT^2))$ is induced from the map $\alpha(x,y)=(x+2\mu,y+2\nu)$.
\end{proof}

\begin{rmk}\label{rmk:line-bundle}
 As a left module over $C(\TT^2)$, $M^c$ corresponds to the module denoted by $X(1,-c)$ in \cite[Section 3.7]{Rie2}.
  The same reference shows that one can construct the corresponding complex line bundle $M^c$ over $\TT^2$ by using Mayer-Vietoris sequences for $K$-theory and the ``clutching'' construction for vector bundles. (Details can be found in \cite[p.298]{Rie2}).
\end{rmk}

\section{Twisted Groupoid $C^*$-algebras}\label{section_twist}

One of the main results in \cite{DKM} is that for each twist $\Lambda$ over $\Gamma$, the groupoid $C^*$-algebra $C^*(\Gamma;\Lambda)$ is isomorphic to the Cuntz-Pimsner algebra $\mathcal{O}_M$, where the $C^*$-correspondence $M$ is constructed from the twist $\Lambda$ (see \cite[Theorem~3.3]{DKM} for further details). Since our main theorem (Theorem~\ref{thm_main}) heavily depends on this result, we give a brief review of this work below.

Let $X$ be a second countable, locally compact, Hausdorff space, let $\sigma:X\rightarrow X$ be a local homeomorphism and let
\[\Gamma=\Gamma(X,\sigma)=\{(x,k-l,y)\in X\times\Z\times X : \sigma^k(x)=\sigma^l(y)\;\;\text{for}\;\;k,l\ge 0\}.\]
The operations on $\Gamma$ are given as follows:
\[r(x,m,y)=x,\;\;s(x,m,y)=y,\;\;(x,m,y)(y,n,z)=(x,m+n,z)\]
\[(x,m,y)^{-1}=(y,-m,x).\]

It is shown in \cite[Theorem~1]{D1} that $\Gamma$ is an \'{e}tale groupoid with the Haar system which is given by the counting measures on the set $r^{-1}(x)$ for $x\in X$.

\begin{rmk}
When $\sigma$ is a homeomorphism, it is routine to check that $\Gamma$ is the transformation groupoid, $X\rtimes_\sigma \ZZ$.
\end{rmk}

Also, it is shown in \cite{DKM} that the groupoid $C^*$-algebra $C^*(\Gamma)$ is isomorphic to the Cuntz-Pimsner algebra associated with the $C^*$-correspondence over $C_0(X)$, which is denoted by $\ell^2(\sigma)$. (The case for $X$ is compact is established in \cite[Proposition~3.3]{D2}). The $C^*$-correspondence $\ell^2(\sigma)$ is described as follows.

Let $\ell^2(\sigma)$ be the completion of the right pre-Hilbert $C_0(X)$-module $C_c(X)$ equipped with the following action and inner-product: for $\xi,\eta\in C_c(X)$ and $g\in C_0(X)$,
\[
\xi\cdot g(x)=\xi(x)g(\sigma(x))\quad\text{and}\quad \langle\xi,\eta\rangle(x)=\sum_{\sigma(y)=x}\overline{\xi}(y)\eta(y).
\]
The left action $\phi$ of the algebra $C_0(X)$ on $\ell^2(\sigma)$ is given by, for $g\in C_0(X)$ and $\xi\in C_c(X)$,
\[
(\phi(g) \cdot \xi)(x)=g(x)\xi(x).
\]
Then $\ell^2(\sigma)$ becomes a $C^*$-correspondence over $C_0(X)$ and we have $C^*(\Gamma)\cong \mathcal{O}_{\ell^2(\sigma)}$.

\begin{dfn}\cite{Ren1}\label{def_twist}
Let $\Gamma$ be an \'{e}tale groupoid. A twist over $\Gamma$ is a topological groupoid $\Lambda$ with $\Lambda^0=\Gamma^0$ such that
the following conditions are satisfied:
\begin{itemize}
\item[(a)] There is a sequence of unit space preserving groupoid maps
\[\Gamma^0\times\TT\xrightarrow[\iota]\quad\Lambda\xrightarrow[\pi]\quad \Gamma,\]
where $\iota$ is injective and $\pi$ is surjective, and $\TT$ is the circle group.
\item[(b)] For $\lambda\in\Lambda$ and $z\in\TT$, $\lambda \iota(s(\lambda),z)=\iota(r(\lambda),z)\lambda$.
\item[(c)]  If $\pi(\lambda_1)=\pi(\lambda_2)$ then $\lambda_2=i(r(\lambda_1),z)\lambda_1$ for some $z\in\TT$.
\end{itemize}
\end{dfn}
It follows that $\TT$ acts freely on $\Lambda$ with quotient $\Lambda/\TT\cong\Gamma$, thus $\Lambda$ can be viewed as a principal $\TT$-bundle over $\Gamma$.

\begin{rmk}\label{rmk:comp}
Let $\Gamma$ be the transformation groupoid $X\rtimes_{\alpha}\ZZ$ associated to the compact Hausdorff space $X$ and homeomorphism $\alpha:X\to X$. The unit space space $\Gamma^{(0)}$ of $\Gamma$ is given by $\Gamma^{(0)}=X\times\{0\},$  which can be identified with $X.$ For $(x,n)\in\Gamma$, the range and source maps are given by $r((x,n))=x$ and $s((x,n))=\alpha^n(x)$. The composition and inverse formulas are given by $(x,n)\cdot (\alpha^n(x),m)=(x,n+m)$ and $(x,n)^{-1}=(\alpha^n(x),-n)$.
\end{rmk}

Given a twist $\Lambda$ over $\Gamma$, we can consider $\Lambda$ as a groupoid in its own right, and form its groupoid $C^*$-algebra $C^*(\Lambda)$, which is the completion of $C_c(\Lambda)$ equipped with the standard convolution. The twisted groupoid $C^*$-algebra $C^*(\Gamma;\Lambda)$
 is defined in \cite{DKM} to be the closure in $C^*(\Lambda)$ of $\{g\in C_c(\Lambda)\mid g(z\lambda)=zg(\lambda)\;\text{for}\; z\in\TT, \lambda\in\Lambda\}$.

\begin{rmk}
 A twist is called a topological twist in \cite{Ku2}, a twisted groupoid in \cite{Ren1}, and  a $\TT$-groupoid in \cite{MW}.
The corresponding constructions of (twisted) groupoid $C^*$-algebras are a bit different from one another, and readers are referred to Appendix A of \cite{aHCR} for details.
\end{rmk}

It is shown in \cite[Theorem~3.3]{DKM} that $C^*(\Gamma;\Lambda)$ is isomorphic to the Cuntz-Pimsner algebra associated to $\mathcal{H}$,  where $\mathcal{H}$ is the $C^*$ correspondence described in the paragraphs that follow.

Let $T$ be a circle bundle over $X$ with $p: T\rightarrow X$ and let $T\times_{\TT}\CC$ be the associated complex line bundle given by $T\times\CC/\negthickspace\sim$, where $(t,w)\sim (zt,zw)$ for $(t,w)\in T\times\CC$ and $z\in \TT$.
Then the space of continuous sections of $T\times_{\TT}\CC$ that vanish at infinity on $T$ denoted by $L_T$ may be identified with
\[\{f:T\rightarrow \CC\mid f(zt)=zf(t)\;\;\text{for}\;\;z\in\TT,t\in T\}.\]

In fact, $L_T$ can be viewed a $C_0(X)$-bimodule with the following actions and inner products.
 As a right $C_0(X)$-module,  for $g\in C_0(X)$, $\xi\in L_T$ and $t\in T$, the right action and the inner product are given by
 \[\xi\cdot g(t):=\xi(t)g(p(t)), \quad \langle\xi,\eta\rangle_R(p(t))=\overline{\xi}(t)\eta(t).\]
 For $g\in C_0(X)$, $\xi\in L_T$ and $t\in T$, the left action $\phi_L$ of $C_0(X)$ and the left inner product are given by
 \[(\phi_L(g)\cdot \xi)(t)=g(p(t))\xi(t)=\xi(t)g(p(t)),\]
 \[\langle\xi,\eta\rangle_L(p(t))=\xi(t)\overline{\eta}(t).\]
 It is straightforward to check the conditions in Definition \ref{symm-module} to see that $L_T$ is a symmetric Hilbert $C_0(X)$-bimodule.

 Since $\ell^2(\sigma)$ is a $C^*$-correspondence over $C_0(X)$, the tensor product $\mathcal{H}:=L_T\otimes_{C_0(X)}\ell^2(\sigma)$ is a right Hilbert $C_0(X)$-module.
Moreover, $\mathcal{H}$ can be viewed as the completion of the compactly supported sections in $L_T$ with the following structure: for $\xi,\eta\in L_T$ and $g\in C_0(X)$,
\[
\langle\xi,\eta\rangle_R(x)=\sum_{\sigma(p(t))=x}\overline{\xi}(t)\eta(t),
\]
\[
\xi\cdot g(t)=\xi(t)g(\sigma(p(t)))\quad\text{and}\quad (\phi_{\mathcal{H}}(g)\cdot\xi)(t)=g(p(t))\xi(t).
\]

Then \cite[Theorem~3.3]{DKM} gives that $\mathcal{O}_{\mathcal{H}}$ is isomorphic to $C^*(\Gamma;\Lambda_T)$, where $\Lambda_T$ is the twist over $\Gamma$ determined by the circle bundle $T$ over $X$. According to \cite[Theorem 3.1]{DKM}, there is one-to-one correspondence between the isomorphism classes of the twists $\Lambda_T$ over $\Gamma$ and the isomorphism classes of the principle circle bundles $T$ over $X$.
In particular, if we define a map $j: X\to\Gamma$ by $j(x)=(x,1,\sigma(x))$, then for a twist $\Lambda$ over $\Gamma$ the corresponding circle bundle $T$ over $X$ can be obtained by $j^*(\Lambda)$, the pull-back of $\Lambda$ by $j$. Also the proof of \cite[Theorem 3.1]{DKM} shows how to construct the twist $\Lambda_T$ explicitly from the given circle bundle $T$ over $X$.

Now we state the main theorem of this paper as follows.

\begin{thm}\label{thm_main}
Let $A$ be a unital commutative $C^*$-algebra. Let $X$ be the spectrum of $A$. Let $\alpha: X\to X$ be a homeomorphism, and let $\Gamma$ be the transformation groupoid $X\rtimes_{\alpha}\ZZ$. Suppose that  $T$ is a circle bundle over $X$, that $\Lambda_T$ is a twist over $\Gamma$ such that $j^*(\Lambda)\cong T$, and that $L_T$ is the space of continuous sections of $T\times_{\TT}\CC$ that vanish at infinity on $T$.
Then
\[
A\rtimes_M \ZZ \cong C^*(\Gamma;\Lambda_T),
\]
where $M\cong L_T\otimes A_{\alpha'}$ and $\alpha'\in \Aut(A)$ is induced from $\alpha$.
\end{thm}
\begin{proof}
 First \cite[Theorem 3.3]{DKM} gives that $ C^*(\Gamma;\Lambda_T)\cong \mathcal{O}_M $.
Since $M$ is a Hilbert $A$-bimodule, we have $\mathcal{O}_M \cong A \rtimes_M\ZZ$ by \cite[Proposition~3.7]{Kat2}. This proves the desired result.
\end{proof}

To apply Theorem~\ref{thm_main} to quantum Heisenberg manifolds, we need to find a circle bundle $T$ over $\TT^2$ such that $L_T$ is isomorphic to the symmetric $C(\TT^2)$-bimodule $M^c$ of Proposition~\ref{prop:D1-module}. It turns out that we can recover $M^{-c}$ from the continuous $\CC$-valued functions $f$ on the Heisenberg manifold $N_{c}$ satisfying $f(zw)=zf(w)$ for $z\in \TT$ and $w\in N_{c}$. (See Proposition~\ref{prop:HM}).

 Let $H$ be the Heisenberg group, parametrized by
 \[
 (x,y,s)=\left (\begin{matrix} 1& y& s\\ 0 & 1& x\\ 0& 0& 1\end{matrix}\right ) \;\;\text{for } x,y,s\in \RR.
 \]
 When we identify $H$ with $\RR^3$, the product is given by
 \[
 (x,y,s)(x',y',s')=(x+x', y+y',s+s'+yx').
 \]

For any positive integer $c$, let $H_c=\{(x,y,z)\in H\mid x,y,cz\in\ZZ\}$. Then the Heisenberg manifold $N_c$ is the quotient $H/H_c$, on which $H$ acts on the left. Thus we can view $N_c$ as a circle bundle over $\TT^2$.
Now fix a positive integer $c$ and we reparametrize the Heisenberg group $H$ as
\[
(x,y,s)=\left( \begin{matrix} 1& y& s/c\\ 0& 1& x\\ 0& 0& 1\end{matrix}\right)\;\;\text{for } x,y,s\in \RR.
\]
Then the product on $\RR^3$ becomes
\[
(x,y,s)(x',y',s')=(x+x',y+y', s+s'+cyx'),
\]
and $H_c=\{(x,y,s)\in H\mid x,y,s\in \ZZ\}$. Thus for $(x,y,s)\in H$ and $(k,m,n)\in H_c$, we have
\[
(x,y,s)(k,m,n)=(x+k,y+m, s+n+cky).
\]
Hence for $f\in C(\RR^3)$, the right translation of $f$ by $(k,m,n)\in H_c$ is given by
\[
f((x,y,s)\cdot (k,m,n))=f(x+k, y+m, z+n+cky).
\]

\begin{prop}\label{prop:HM}
Suppose $M^c$ is the symmetric $C(\TT^2)$-bimodule of Proposition~\ref{prop:D1-module} and $N_c$ is the Heisenberg manifold defined as above. Then
\[
M^c\cong \{f:N_{-c}\to \CC \mid f(zw)=zf(w)\;\text{for}\; z\in \TT, w\in N_c\}.
\]
\end{prop}

\begin{proof}
Since the continuous functions on $N_c$ are invariant under the right action by $H_c$, $f\in C(N_c)$ should satisfy, for $(k,m,n)\in H_c$,
\[
f((x,y,s)\cdot (k,m,n))=f(x,y,s).
\]
In particular, $f\in C(N_c)$ satisfies
\begin{itemize}
\item[(i)] $f(x,y,s)=f((x,y,s)\cdot (1,0,0))=f(x+1,y, s+cy)$,
\item[(ii)] $f(x,y,s)=f((x,y,s)\cdot (0,1,0))=f(x,y+1,s)$,
\item[(iii)] $f(x,y,s)=f((x,y,s)\cdot (0,0,1))=f(x,y,s+1)$.
\end{itemize}

Let $z=e^{2\pi it}$ for some $t\in \RR$. Define an action of $\TT$ on $N_c$ by
\[
z\cdot (x,y,s)=(x,y,s+t).
\]
 Then for $z\in \TT$ and $w\in N_c$, the condition $f(zw)=zf(w)$ gives
\begin{itemize}
\item[(iv)] $f(x,y,s+t)=e^{2\pi it}f(x,y,s)$.
\end{itemize}
Set $s=0$. Then (i) and (iv) give
\begin{itemize}
\item[(v)] $f(x,y,0)=f(x+1,y,cy)=e^{2\pi i cy} f(x+1,y,0)$.
\end{itemize}
Also (ii) and (iii) give
\begin{itemize}
\item[(vi)] $f(x,y+1,0)=f(x,y,0)=f(x,y,1)$.
\end{itemize}
Now define $F:\RR^2\to \CC$ by $F(x,y)=f(x,y,0)$. Then we have
\[\begin{split}
&\{f:N_c\to \CC\mid f(zw)=zf(w)\;\;\text{for}\; z\in \TT, w\in N_c\}\\
&\cong\{F:\RR^2\to \CC\mid F(x,y+1)=F(x,y), F(x+1,y)=e^{-2\pi icy}F(x,y)\}\\
&=\{F:\RR\times\TT \to \CC \mid F(x+1,y)=e^{-2\pi icy}F(x,y)\}=M^{-c}.
\end{split} \]
 Thus replacing $c$ by $-c$ gives the desired result.
\end{proof}

By combining Proposition~\ref{prop:HM} and \cite[Theorem~3.1]{DKM}, we obtain the following result for the quantum Heisenberg manifold as a corollary of Theorem~\ref{thm_main}.

\begin{cor}\label{Cor:main1}
Let $D^{c}_{\mu\nu}$ be the quantum Heisenberg manifold. Let $\alpha$ be the homeomorphism on $\TT^2$ defined by $\alpha(x,y)=(x+2 \mu,y+2 \nu)$ and $\Gamma= \TT^2\rtimes_{\alpha}\ZZ$. Then there exists a twist $\Lambda$ over $\Gamma$ such that
\[
D^{c}_{\mu\nu}\cong C^*(\Gamma;\Lambda).
\]
\end{cor}

\begin{proof}
By Proposition~\ref{prop:D1-module}, we have that $D^{c}_{\mu,\nu}\cong C(\TT^2)\rtimes_{D_1}\ZZ$ and $D_1\cong M^c\otimes (C(\TT^2))_{\alpha'}$ where $\alpha'\in\Aut(C(\TT^2))$ is induced from the map $\alpha$.  Proposition~\ref{prop:HM} gives the circle bundle $N_{-c}$ over $\TT^2$ such that $M^c$ is isomorphic to the space of continuous sections of $N_{-c}\times_{\TT}\CC$ that vanish at infinity on $N_{-c}$. Then \cite[Theorem~3.1]{DKM} gives a twist $\Lambda$ such that $N_{-c}\cong j^*(\Lambda)$, where $j:\TT^2\to \Gamma$ is defined by $j(t)=(t,1,\alpha(t))$. Hence Theorem~\ref{thm_main} gives the desired result.
\end{proof}

 As a refinement of the construction described in \cite[Theorem~3.1]{DKM}, in the next section, we give a specific construction of a twist groupoid obtained from a transformation groupoid and a principal bundle over its unit space, by means of transition functions of the bundle.  This is meant to give a topological insight into the notion of twist groupoids.

\section{Construction of a groupoid from a $\mathbb Z$-transformation groupoid and a principal bundle over its unit space}

Let $X$ be a compact metric space, and let $\alpha:X\to X$ be a homeomorphism of $X$ onto itself.  Let the associated transformation groupoid
$\Gamma= X\times_{\alpha} \mathbb Z$ be as defined in Remark~\ref{rmk:comp}. Recall then that $(x,m)$ and $(y,n)$ are composable if and only if $y=\alpha^{m}(x)$ and the composition formula is given by $(x,m)\cdot (\alpha^m(x),n)\;=\;(x,m+n)$.

Let $G$ be a compact abelian group and $(E,p,X)$ be a locally trivial principal $G$-bundle with base space $X$, where $G$ acts freely and properly on $E$ on the left.
We will use this information to form a principal $G$-bundle over the groupoid $\Gamma$ that itself has a groupoid structure.  When $G=\TT,$ our construction produces the twist groupoid  \cite[Theorem~3.1]{DKM}, which was constructed via line bundles.  Although the construction that follows is more technical and involves more book-keeping, it can be applied to this more general setting.

The quotient space for the $G$-action on $E$ is exactly given by $X$ with the quotient map $p:E\to X$.
It is well-known (cf. \cite{Hus}) that one can construct from this set up a finite open cover $\{U_j: j\in \mathbb J\}$ of $X$ and transition functions $c_{i,j}:U_i\cap U_j\to G$ satisfying
\begin{itemize}
\item[(a)] $c_{i,i}(x)\equiv 1_G$ for $i \in\mathbb J;$
\item[(b)] $c_{i,j}(x)=[c_{j,i}(x)]^{-1}$ for $ i, j\in \mathbb J$ and $x\in U_i\cap U_j;$
\item[(c)] $c_{i,j}(x)c_{j,k}(x)=c_{i,k}(x)$ for $i, j, k\in \mathbb J$ and $x\in U_i\cap U_j\cap U_k.$
\end{itemize}
Moreover, up to equivalence of principal bundles, the bundle $E$ can be reconstructed from the finite open cover $\{U_j\}$ and the transition functions $\{c_{i,j}\}$ as follows.

Let $\Omega = \bigsqcup_{i\in \mathbb J} G\times U_i\times\{i\}$ and define an equivivalence relation on $\Omega$ by
\begin{equation}\label{eq:bundle}
(g_1,b_1,i)\sim(g_2,b_2,j)\;\text{if}\;\; b_1=b_2\in U_i\cap U_j,\;\;\text{and}\;\;g_1c_{i,j}(b_1)=g_2.
\end{equation}
Then $\Omega$ is a $G$-bundle over $ \bigsqcup_{i\in \mathbb J}U_i$. Setting $\widetilde{\Omega}$ to be the quotient space $\Omega/\negthickspace\sim,$  we obtain that $\widetilde{\Omega}$ is a principal $G$-bundle over $X$ that is equivalent to $E.$

Therefore, without loss of generality we will construct the principal $G$-bundle $\Lambda$ over $\Gamma$ by means of transition functions.
Recall that $\Gamma$ is locally compact and can be written as a disjoint union of compact sets, $\Gamma\;=\;\bigsqcup_{n\in\mathbb Z}X\times \{n\}.$
For each $n,$ we will construct a principal $G$-bundle $\Lambda_n$ over $X\times \{n\}=X$ by means of transition functions. The final bundle $\Lambda$ will be equal to $\bigsqcup_{n\in\mathbb Z} \Lambda_n.$

We define $\Lambda_0=G\times X\times \mathbb{J}^0\times \{0\}$, where the index set $\mathbb{J}^0:=\{1\}$ corresponds to the trivial open covering $\{X\}$ for the trivial $G$-bundle over $X$ and the trivial transition function $\{c^{0}_{1,1}(x)\equiv 1_G\}.$

Sitting inside $\Lambda_0$ we have the unit space of all
$\Lambda:$
\[\Lambda^{(0)}\;=\{(1_G,x,0): x\in X\}\;\cong X\times\{0\}=\Gamma^{(0)}.\]
The embedding of $\Lambda^{(0)}$ into $\Lambda_0$ is given by
$(1_G,x,0)\mapsto (1_G,x,1,0)\in\Lambda_0,$ where $1\in \J^0$. 

Given the finite open cover $\{U_i\}_{i\in \mathbb J}$ and transition functions $\{c_{i,j}=c^{(1)}_{i,j}:U_i\cap U_j\to G\}$ determining the principal bundle $(E, p, X)$ up to bundle equivalence, we define
$$\Lambda_1\;=\;\widetilde{\Omega}\times \{1\}.$$ The projection from $\widetilde{\Omega}$ onto $X$ extends to a projection
$\Pi_1$ from $\Lambda_1$ onto $X\times \{1\},$ which coincides with the quotient map for the $G$-action.

To define $\Lambda_n$, we need to take a refinement of the covering $\{U_i\}_{i\in\J}$. For a positive integer $n$, set $\J^n=\overbrace{\J\times\dots\times\J}^{\text{$n$-times}}$.  We form a finite open cover $\{U_{\vec{i}}\}_{\vec{i}\in\J^n}$ given by
\begin{equation}\label{eq:positive_oc}
U_{\vec{i}}=U_{\pi_1(\vec{i})} \cap \alpha^{-1}(U_{\pi_2(\vec{i})})\cap \cdots \cap  \alpha^{-(n-1)}(U_{\pi_n(\vec{i})})=\cap_{j=1}^{n}\alpha^{-(j-1)}(U_{\pi_j(\vec{i})}),
\end{equation}
where $\pi_j:\J^n\to \J$ is the projection map onto the $j^{th}$ variable.

For $\vec{i}$ and $\vec{j}\in {\mathbb J}^n,$ the transition functions $c^{(n)}_{\vec{i},\vec{j}}:U_{\vec{i}}\cap U_{\vec{j}}\to G$ are defined by
\begin{equation}\label{eq:positive_tf}
\begin{split}
c^{(n)}_{\vec{i},\vec{j}}(x)&=c_{\pi_1(\vec{i}),\pi_1(\vec{j})}(x)\cdot c_{\pi_2(\vec{i}),\pi_2(\vec{j})}(\alpha(x))\cdot \cdots \cdot c_{\pi_n(\vec{i}),\pi_n(\vec{j})}(\alpha^{n-1}(x))\\
&=\;\Pi_{k=1}^nc_{\pi_k(\vec{i}),\pi_k(\vec{j})}(\alpha^{k-1}(x))\;\;\text{for}\;\; x\in U_{\vec{i}}\cap U_{\vec{j}}.
\end{split}\end{equation}
One verifies that the $\{c^{(n)}_{\vec{i},\vec{j}}\}$ are transition functions for the open cover $\{U_{\vec{i}}\}_{\vec{i}\in \mathbb J^n}$.

The finite open cover $\{U_{\vec{j}}\}_{\vec{j}\in {\mathbb J}^n}$ and transition functions $\{c^{(n)}_{\vec{i},\vec{j}}\}$ defined on intersections again produce a principal $G$-bundle $\widetilde{\Omega}_n$ over $X,$ and we define
$\Lambda_n=\widetilde{\Omega}_n\times\{n\}.$ As with the case $n=1,$ the projection from $\widetilde{\Omega}_n$ onto $X$ extends to a projection
$\Pi_n$ from $\Lambda_n$ onto $X\times \{n\},$ which again coincides with the quotient map for the $G$-action. As in \eqref{eq:bundle}, for $n>0$, $(g,x,\vec{i},n) \in \Lambda_n$ is equivalent to $(h,x,\vec{j},n)\in \Lambda_n$ if $x\in U_{\vec{i}}\cap U_{\vec{j}}$ and $h=c^n_{\vec{i},\vec{j}}(x) g.$ 
We are left with finding the bundles over $X\times \{n\}\subseteq \Gamma$ in the case where $n<0.$  We start with $n=-1.$

Using the same index set $\mathbb J$ and same open cover $\{U_i\}_{i\in \mathbb J}$ for $X$ as before, set
$W_i=\alpha(U_i)$. Considering the open cover $\{W_i\}_{i\in \mathbb J},$ define transition functions $c^{(-1)}_{i,j}:W_i\cap W_j\to G$ by
\[
c^{(-1)}_{i,j}(x)\;=\;[c_{i,j}(\alpha^{-1}(x))]^{-1}\;\text{for}\; x\in W_i\cap W_j = \alpha(U_i\cap U_j).
\]
One verifies that $\{c^{(-1)}_{i,j}\}$ are transition functions for the open cover $\{W_i\}_{i\in \mathbb J},$ and this finite open cover and transition function pair gives a principal $G$-bundle $\widetilde{\Omega}_{-1}$ over $X,$ and we define
$\Lambda_{-1}=\widetilde{\Omega}_{-1}\times\{-1\}.$ As in the case with non-negative indices, there is a projection $\Pi_{-1}$ from $\Lambda_{-1}$ onto $X\times \{-1\}\subseteq \Gamma.$

Now for general negative integer $n$, let our index set be ${\mathbb J}^{|n|}$. For $\vec{i}\in{\mathbb J}^{|n|}$
define the open set $W_{\vec{i}}\subset X$ by
\begin{equation}\label{eq:negative_oc}
W_{\vec{i}}\;=\;W_{\pi_1(\vec{i})}\cap \alpha(W_{\pi_2(\vec{i})})\cap \cdots \cap  \alpha^{(|n|-1)}(W_{\pi_{|n|}(\vec{i})})=\cap_{k=1}^{|n|}\alpha^{(k-1)}(W_{\pi_k(\vec{i})}).
\end{equation}
One checks in the standard way that $\{W_{\vec{i}}\}_{\vec{i}\in {\mathbb J}^{|n|}}$ is a finite open cover for $X.$

 For $\vec{i}$ and $\vec{j}\in {\mathbb J}^{|n|},$ the transition functions $c^{(n)}_{\vec{i},\vec{j}}:W_{\vec{i}}\cap W_{\vec{j}}\to G$ are defined by
\begin{equation}\begin{split}\label{eq:negative_tf1}
c^{(n)}_{\vec{i},\vec{j}}(x)&=c^{(-1)}_{\pi_1(\vec{i}),\pi_1(\vec{j})}(x)\cdot c^{(-1)}_{\pi_2(\vec{i}),\pi_2(\vec{j})}(\alpha^{-1}(x))\cdot \cdots \cdot c^{(-1)}_{\pi_n(\vec{i}),\pi_n(\vec{j})}(\alpha^{-(|n|-1)}(x))\\
&=\;\Pi_{k=1}^{|n|}c^{(-1)}_{\pi_k(\vec{i}),\pi_k(\vec{j})}(\alpha^{-(k-1)}(x))\;\;\text{for}\;\; x\in W_{\vec{i}}\cap W_{\vec{j}}.
\end{split}\end{equation}

For $n<0$ and $x\in W_{\vec{i}}\cap W_{\vec{j}}$, one can calculate that
\begin{equation}\label{eq:negative_tf2}
c^{(n)}_{\vec{i},\vec{j}}(x)=\;\Pi_{k=1}^{|n|}[c_{\pi_k(\vec{i}),\pi_k(\vec{j})}(\alpha^{-k}(x))]^{-1}.
\end{equation}
As before, the finite open cover $\{W_{\vec{j}}\}_{\vec{j}\in {\mathbb J}^{|n|}}$ and transition functions $\{c^{(n)}_{\vec{i},\vec{j}}\}$ defined on intersections again produces a principal $G$-bundle $\widetilde{\Omega}_n$ over $X,$ and we define
$\Lambda_n=\widetilde{\Omega}_n\times\{n\}.$ The projection from $\widetilde{\Omega}_n$ onto $X$ again extends to a projection
$\Pi_n$ from $\Lambda_n$ onto $X\times \{n\},$ which coincides with the quotient map for the $G$ action. Also, for $n<0$, $(g,x,\vec{i},n) \in \Lambda_n$ is equivalent to $(h,x,\vec{j},n)\in \Lambda_n$ if $x\in W_{\vec{i}}\cap W_{\vec{j}}$ and $h=c^n_{\vec{i},\vec{j}}(x) g.$

For future reference, we note that for $n>0$ we can define the ``flip'' involution $\sigma_n:{\mathbb J}^n\to {\mathbb J}^n$ to reverse the order of the indices, i.e. define $\sigma_n$ by
\begin{equation}\label{eq:flip}
\sigma_n(i_1,i_2,\cdots, i_n)=(i_n,i_{n-1},\cdots, i_1).
\end{equation}
Then for $n>0$, $\vec{i},\;\vec{j}\in {\mathbb J}^n,$ and $x\in U_{\vec{i}}\cap U_{\vec{j}}$, one can show that
\begin{equation}\label{eq:flip_tf1}
c^{(n)}_{\vec{i},\vec{j}}(x)=c^{(-n)}_{\sigma_n(\vec{j}),\sigma_n(\vec{i})}(\alpha^n(x)).
\end{equation}
Similarly, for $n<0$ and $\vec{i},\;\vec{j}\in {\mathbb J}^{-n},$ if $x\in W_{\vec{i}}\cap W_{\vec{j}}$, one can show that
\begin{equation}\label{eq:negative_tf3}
c^{(n)}_{\vec{i},\vec{j}}(x)=c^{(-n)}_{\sigma_{|n|}(\vec{j}),\sigma_{|n|}(\vec{i})}(\alpha^{-|n|}(x)).
\end{equation}

  We note that for $n_1>n_2>0,$ the finite open cover $\{U_{\vec{i}}\}_{\vec{i}\in \mathbb J^{n_1}}$ of $X$ is a refinement of the finite open cover $\{U_{\vec{j}}\}_{\vec{j}\in \mathbb J^{n_2}}$ of $X,$ where for $\vec{i}\in \mathbb J^{n_1}$ the embedding is given by $U_{\vec{i}}\subseteq  U_{\Pi_{n_2}^{n_1}(\vec{i})}$ and the map $\Pi_{n_2}^{n_1}: {\mathbb J^{n_1}}\to {\mathbb J^{n_2}}$ is given by
\[
\Pi_{n_2}^{n_1}(\vec{i})=\;(\pi_1(\vec{i}),\pi_2(\vec{i}),\cdots, \pi_{n_2}(\vec{i}))\in {\mathbb J^{n_2}}.
\]

Similarly, if $n_1<n_2<0,$ one can show using the map $\Pi_{|n_2|}^{|n_1|}: {\mathbb J^{|n_1|}}\to {\mathbb J^{|n_2|}}$ that  the finite open cover $\{W_{\vec{i}}\}_{\vec{i}\in \mathbb J^{|n_1|}}$ of $X$ is a refinement of the finite open cover $\{W_{\vec{j}}\}_{\vec{j}\in \mathbb J^{|n_2|}}$ of $X.$

We are ready to proceed with the definition of our twist groupoid.
Let $\Lambda=\bigsqcup_{n\in \mathbb Z}\Lambda_{n}.$  The map $\Pi:\Lambda\to \Gamma=X\times \mathbb Z$ is defined piecewise on each $\Lambda_n$ by $\Pi_{|_{\Lambda_n}}=\Pi_n.$  The group $G$ acts on each `component' $\Lambda_n$ with quotient $X\times\{n\};$ in the case where $X$ and $G$ are connected, these are the actual components of $\Lambda.$ Recall that by definition of twist groupoids, the unit space of $\Lambda$ is given by $\Lambda^{(0)}\;=\{(1_G,x,0): x\in X\}\;\cong X\times\{0\}=\Gamma^{(0)}.$

For $n\geq 0$ and $(g,x, \vec{i}, n)\in \Lambda_n$ with $g\in G,\;\vec{i}\in {\mathbb J}^n,\;x\in  U_{\vec{i}}$, we have the range and source maps $r, s:\Lambda_n\to \Lambda^{(0)}$ given by
\[
r((g,x, \vec{i}, n))\;=\;(1_G,x,0),\;\;s((g,x, \vec{i}, n))\;=\;(1_G,\alpha^n(x),0).
\]
For $n<0$ and $(g,x, \vec{i}, n)\in \Lambda_n$ with $g\in G,\;\vec{i}\in {\mathbb J}^{|n|},\;x\in  W_{\vec{i}},$ we have the range and source maps $r, s:\Lambda_n\to \Lambda^{(0)}$  given by
\[
r((g,x, \vec{i}, n))\;=\;(1_G,x,0),\;\;s((g,x, \vec{i}, n))\;=\;(1_G,\alpha^n(x),0).
\]
One easily checks that these range and source maps are well-defined on equivalence classes in each $\Lambda_n$ determined by the transition functions.

In the following theorem, we define the groupoid product $(g,x,\vec{i},n)\cdot (g',x',\vec{i'},n')$
of $(g,x,\vec{i},n)\in \Lambda_n$ with $(g',x',\vec{i'},n')\in \Lambda_{n'}$ when $r((g,x,\vec{i},n))=s((g',x',\vec{i'},n')),$
i.e. $x'=\alpha^n(x).$

\begin{thm}\label{thm:twist_comp}
Let $X$, $G$, $\alpha$, $\Gamma$ and $\Lambda$ be as above with $\Lambda=\bigsqcup_{n\in\mathbb Z} \Lambda_n$.
Then $\Lambda$ is a topological groupoid, and its unit space $\Lambda^{(0)}\subset \Lambda_0$ can be identified with $X\times \{0\}=\Gamma^{(0)}$.
If $(g,x,\vec{i},n)\in \Lambda_n\subset \Lambda$ and $(g',x',\vec{i'},n')\in \Lambda_{n'}\subset \Lambda$ satisfy
$s((g,x,\vec{i},n))=r((g',x',\vec{i'},n')),$ i.e.  $\alpha^n(x)=x'$, then the formulas for the groupoid product are given case by case as follows.
\begin{itemize}
\item[(a)] If $n,\;n'\geq 0$, then $[(g,x,\vec{i},n)]\cdot [(g',x',\vec{i'},n')]= [(gg', x, (\vec{i},\vec{i'}), n+n')]$.
\item[(b)] If $n,\;n'\; <0$, then $[(g, x, \vec{i}, n)]\cdot [(g',x'=\alpha^n(x), \vec{i'}, n')]= [(gg', x, (\vec{i},\vec{i'}), n+n')]$.
\item[(c)] If $n>0$ and $n'=-n<0$, then \[[(g, x, \vec{i}, n)]\cdot [(g',x'=\alpha^n(x), \vec{i'}, n')]=[(c^{(n)}_{\vec{i},\sigma_n(\vec{i'})}(x)gg', x, 1,0)].\]
\item[(d)] If $n<0$ and $n'=-n >0$, then \[[(g, x, \vec{i}, n)]\cdot [(g',x'=\alpha^n(x), \vec{i'}, n')]=[(c^{(-n)}_{\vec{i'},\sigma_{-n}(\vec{i})}(\alpha^{n}(x))gg', x, 1, 0)].\]
\item[(e)] If  $n<0<n'<|n|,$ and $\vec{i}=(\vec{i_1},\vec{i_2})$ with $\vec{i_1}\in {\mathbb J}^{-n-n'}$ and $\vec{i_2}\in {\mathbb J}^{n'}$, then \[[(g,x,\vec{i},n)]\cdot [(g',\alpha^n(x),\vec{i'},n')]=\;[(c^{(n')}_{\vec{i'},\sigma_{n'}(\vec{i_2})}(\alpha^n(x))gg', x, \vec{i_1}, n+n')].\]
\item[(f)] If $n<0<n'$ with $|n|< n',$ and $\vec{i'}=\;(\vec{i_1'}, \vec{i_2'})\in {\mathbb J}^{n'}$ with $\vec{i_1'}\in {\mathbb J}^{-n}$ and $\vec{i_2'}\in {\mathbb J}^{n+n'},$ then \[[(g,x,\vec{i},n)]\cdot [(g',\alpha^n(x),\vec{i'},n')]=\;[(c^{(-n)}_{\vec{i_1'},\sigma_{-n}(\vec{i})}(\alpha^n(x))\cdot g\cdot g', x, \vec{i_2'}, n+n')].\]
\item[(g)] If $|n'| > n>0>n',$ and $(\vec{i_1'},\vec{i_2'})=\vec{i'}\in {\mathbb J}^{-n'}$ with
 $\vec{i_1'}\in  {\mathbb J}^{n},\; \vec{i_2'}\in {\mathbb J}^{-n'-n}$, then \[[(g,x,\vec{i},n)]\cdot [(g',\alpha^n(x),\vec{i'},n')]\;=\;[(c^{(n)}_{\vec{i},\sigma_n(\vec{i_1'})}(x)\cdot gg', x,  \vec{i_2'}, n+n')].\]
\item[(h)] If $n>|n'|>0>n',$ and  $(\vec{i_1},\vec{i_2})=\;\vec{i}\;\in\;{\mathbb J}^n$ with $\vec{i_1}\in {\mathbb J}^{n+n'}$ and $\vec{i_2}\in {\mathbb J}^{-n'}$, then  \[[(g,x,\vec{i},n)]\cdot [(g',\alpha^n(x),\vec{i'},n')]=\;[(c^{(-n')}_{\vec{i_2},\sigma_{-n'}(\vec{i'})}(\alpha^{n+n'}(x))\cdot gg', x, \vec{i_1}, n+n')].\]
\end{itemize}
Moreover for $[(g, x, \vec{i}, n)]\in \Lambda_n\subset \Lambda,$ the inverse operation maps $\Lambda_n$ to $\Lambda_{-n}$ and is given by:
\[
[(g, x, \vec{i}, n)]^{-1}\;=\; [(g^{-1}, \alpha^n(x), \sigma_{|n|}(\vec{i}),-n)].
\]
\end{thm}

\begin{proof}
We note if $n=0,$ then $x=x',$ and the definition
\[
(g,x,1,0)\cdot (g',x'=x,\vec{i'},n')=\;(gg',x'=x,\vec{i'},n')
\]
preserves equivalence classes. Likewise, if $n'=0,$ and $x'=\alpha^n(x)$, then for $n\in\ZZ$ the definition
\[
(g,x,\vec{i},n)\cdot (g',x'=\alpha^n(x),1,0)=\;(gg', x,\vec{i},n)
\]
preserves equivalence classes.

For (a), suppose $n,\;n'>0$. For $x\in U_{\vec{i}}$ and $x'\in U_{\vec{i'}}$, suppose $x'=\alpha^n(x)$. Then \eqref{eq:positive_oc} gives
\[\begin{split}
x=\alpha^{-n}(x') & \in  \alpha^{-n}(U_{\pi_1(\vec{i'})}\cap \alpha^{-1}(U_{\pi_2(\vec{i'})})\cap \cdots \cap \alpha^{-(n'-1)}(U_{\pi_{n'}(\vec{i'})})) \\
&=\alpha^{-n}(U_{\pi_1(\vec{i'})})\cap \alpha^{-(n+1)}(U_{\pi_2(\vec{i'})})\cap \cdots \cap  \alpha^{-(n+n'-1)}(U_{\pi_{n'}(\vec{i'})})).
\end{split}\]
It follows that
\[
x\in\;[U_{\pi_1(\vec{i})}\cap \cdots \cap \alpha^{-(n-1)}(U_{\pi_n(\vec{i})})]\cap [\alpha^{-n}(U_{\pi_1(\vec{i'})})\cap \cdots \cap  \alpha^{-(n+n'-1)}(U_{\pi_{n'}(\vec{i'})}))].
\]
This implies that
\[
x\in\;U_{\pi_1((\vec{i},\vec{i'}))}\cap \cdots \cap \alpha^{-(n-1)}(U_{\pi_n((\vec{i},\vec{i'}))})\cap \alpha^{-n}(U_{\pi_{n+1}((\vec{i},\vec{i'}))})\cap \cdots \cap  \alpha^{-(n+n'-1)}(U_{\pi_{n+n'}((\vec{i},\vec{i'}))}),
\]
i.e. $x\in\;U_{(\vec{i},\vec{i'})}.$

Now suppose that $(g,x,\vec{i},n)\sim (h,x,\vec{j},n)\in \Lambda_n$ and $(g',x'=\alpha^n(x),\vec{i'},n')\sim (h',x'=\alpha^n(x),\vec{j'},n')\in \Lambda_{n'}.$
We need to check that
\[(gg', x, (\vec{i},\vec{i'}), n+n')\;\sim\;(hh', x, (\vec{j},\vec{j'}), n+n')\;\in\;\Lambda_{n+n'}.\]
Since $(g,x,\vec{i},n)\sim (h,x,\vec{j},n),$ we know that $x\in U_{\vec{i}}\cap U_{\vec{j}}$ and $h=c^{(n)}_{\vec{i}, \vec{j}}(x)g,$ by \eqref{eq:positive_tf} that is
\[h=\Pi_{k=1}^nc_{\pi_k(\vec{i}),\pi_k(\vec{j})}(\alpha^{k-1}(x))g.\]
Similarly, since $(g',\alpha^n(x),\vec{i'},n')\sim (h',\alpha^n(x),\vec{j'},n'),$ we know that
$\alpha^n(x)\in U_{\vec{i'}}\cap U_{\vec{j'}}$ and $h'=c^{(n')}_{\vec{i'}, \vec{j'}}(\alpha^n(x))g',$ i.e. $h'=\Pi_{k=1}^{n'}c_{\pi_k(\vec{i'}),\pi_k(\vec{j'})}(\alpha^{k-1}(\alpha^n(x)))g',$ that is,
\[h'=\Pi_{k=n+1}^{n+n'}c_{\pi_k(\vec{i'}),\pi_k(\vec{j'})}(\alpha^{k-1}(x))g'.\]
Combining equalities, we get
\[\begin{split}h\cdot h' &=\; \Pi_{k=1}^nc_{\pi_k(\vec{i}),\pi_k(\vec{j})}(\alpha^{k-1}(x))\Pi_{k=n+1}^{n+n'}c_{\pi_k(\vec{i'}),\pi_k(\vec{j'})}(\alpha^{k-1}(x))\; g\cdot g'\\
&=\; \Pi_{k=1}^{n+n'}c_{\pi_k((\vec{i},\vec{i'})),\pi_k((\vec{j},\vec{j'}))}(\alpha^{k-1}(x))g\cdot g'.
\end{split}\]
By definition of the transition functions $c^{(n+n')}_{(\vec{i},\vec{i'}),(\vec{j},\vec{j'})},$ this implies that
\[h\cdot h=\;c^{(n+n')}_{(\vec{i},\vec{i'}),(\vec{j},\vec{j'})}(x)g\cdot g'.\]  Therefore we have shown that
\[(gg', x, (\vec{i},\vec{i'}), n+n')\;\sim\;(hh', x, (\vec{j},\vec{j'}), n+n')\;\in\;\Lambda_{n+n'}.\]
Thus our proposed groupoid product is well-defined on equivalence classes.

For (b), suppose $n<0$ and $n'<0,$ and $x'=\alpha^n(x)$. Also suppose that
$(g,x,\vec{i},n)\sim (h,x,\vec{j},n)\in \Lambda_n$ and $(g',x'=\alpha^n(x),\vec{i'},n')\sim (h',x'=\alpha^n(x),\vec{j'},n')\in \Lambda_{n'}$.
Then $x\in W_{\vec{i}}\cap W_{\vec{j}}$ and $\alpha^n(x)\in W_{\vec{i'}} \cap W_{\vec{j'}}$. Also we have
\[h=c^{(n)}_{\vec{i}, \vec{j}}(x)g\;\;\text{and}\;\;h'=c^{(n')}_{\vec{i'}, \vec{j'}}(\alpha^n(x))g'.\]
 Therefore \eqref{eq:negative_tf2} gives
\[\begin{split}
h\cdot h'\;&=\;c^{(n)}_{\vec{i}, \vec{j}}(x)g\cdot c^{(n')}_{\vec{i'}, \vec{j'}}(\alpha^n(x))g'\\
&=\;\Pi_{k=1}^{|n|}[c_{\pi_k(\vec{i}),\pi_k(\vec{j})}(\alpha^{-k}(x))]^{-1}\cdot \Pi_{k=1}^{|n'|}[c_{\pi_k(\vec{i'}),\pi_k(\vec{j'})}(\alpha^{-k}(\alpha^n(x)))]^{-1}\cdot gg'\\
&=\;\Pi_{k=1}^{|n|}[c_{\pi_k((\vec{i},\vec{i'})),\pi_k((\vec{j},\vec{j'}))}(\alpha^{-k}(x))]^{-1}\cdot
\Pi_{k=1}^{|n'|}[c_{\pi_{|n|+k}((\vec{i},\vec{i'})),\pi_{|n|+k}((\vec{j},\vec{j'}))}(\alpha^{-k}(\alpha^{-|n|}(x)))]^{-1}\cdot gg'\\
&=\;\;\Pi_{k=1}^{|n|}[c_{\pi_k((\vec{i},\vec{i'})),\pi_k((\vec{j},\vec{j'}))}(\alpha^{-k}(x))]^{-1}\cdot
\Pi_{k=|n|+1}^{|n|+|n'|}[c_{\pi_k((\vec{i},\vec{i'})),\pi_k((\vec{j},\vec{j'}))}(\alpha^{-k}(x))]^{-1}\cdot gg'\\
&=\;\Pi_{k=1}^{|n+n'|}[c_{\pi_k((\vec{i},\vec{i'})),\pi_k((\vec{j},\vec{j'}))}(\alpha^{-k}(x))]^{-1}\cdot gg'\\
&=\;c^{(n+n')}_{(\vec{i},\vec{i'}),(\vec{j},\vec{j'})}(x)\cdot gg'.\end{split}\]

Therefore, in the case $n<0$ and $n'<0$ we have again shown that
$$(gg', x, (\vec{i},\vec{i'}), n+n')\;\sim\;(hh', x, (\vec{j},\vec{j'}), n+n')\;\in\;\Lambda_{n+n'},$$
so that our proposed groupoid product is well-defined on equivalence classes in this case as well.

For (c), suppose $n>0$, $n' = -n<0$ and $\alpha^n(x)=x'$.
Since $n'=-n<0,$ \eqref{eq:negative_oc} implies that
\[x'=\alpha^n(x)\in W_{\vec{i'}}= \; W_{\pi_1(\vec{i'})}\cap \alpha(W_{\pi_2(\vec{i'})})\cap \cdots \cap  \alpha^{(|n'|-1)}(W_{\pi_{|n'|}(\vec{i'})}).\]
Since $\{W_i=\alpha(U_i)\}$ and $n'=-n$, we have
\[\alpha^n(x)\in \;\alpha(U_{\pi_1(\vec{i'})})\cap \alpha^2(U_{\pi_2(\vec{i'})})\cap \cdots \cap  \alpha^n(U_{\pi_{|n|}(\vec{i'})}).\]
This in turn implies the inclusion
\[x\in \alpha^{-(n-1)}(U_{\pi_1(\vec{i'})})\cap \alpha^{-(n-2)}(U_{\pi_2(\vec{i'})})\cap \cdots \cap U_{\pi_{n}(\vec{i'})}= U_{\sigma_n(\vec{i'})},\]
where $\sigma_n$ is the ``flip'' involution given in \eqref{eq:flip}.
So we have $\alpha^n(x)\in W_{\vec{i'}}$ if and only if $x\in U_{\sigma_n(\vec{i'})}$.
Therefore, if $x\in U_{\vec{i}},$ and $\alpha^n(x)\in  W_{\vec{i'}},$ then $x\in U_{\vec{i}}\cap U_{\sigma_n(\vec{i'})}.$
By definition of our bundle in \eqref{eq:bundle}, it follows that
\begin{equation}
(g,x,\vec{i},n)\;\sim\;(c^{(n)}_{\vec{i},\sigma_n(\vec{i'})}(x)\cdot g, x, \sigma_n(\vec{i'}), n).
\end{equation}
Now suppose that $(g,x,\vec{i},n)\sim (h,x,\vec{j},n)\in \Lambda_n$ and $(g',x'=\alpha^n(x),\vec{i'},-n)\sim (h',x'=\alpha^n(x),\vec{j'},-n)\in \Lambda_{-n}.$

The proposed formula gives us that the product of $(h,x,\vec{j},n)$ and $(h',x'=\alpha^n(x),\vec{j'},-n)$ should be equal to the equivalence class of $(c^{(n)}_{\vec{j},\sigma_n(\vec{j'})}(x)\cdot h\cdot h', x, 1, 0).$

Since $(g,x,\vec{i},n)\sim (h,x,\vec{j},n),$ we know that $x\in U_{\vec{i}}\cap U_{\vec{j}}$ and $h=c^{(n)}_{\vec{i},\vec{j}}(x)g$.
Since  $(g',\alpha^n(x),\vec{i'},-n)\sim (h',\alpha^n(x),\vec{j'},-n),$ we know that
$\alpha^n(x)\in W_{\vec{i'}}\cap W_{\vec{j'}}$ and $h'=c^{(-n)}_{\vec{i'}, \vec{j'}}(\alpha^n(x))g'$. Thus \eqref{eq:negative_tf2} and \eqref{eq:flip} give
\[\begin{split}
h'&=\Pi_{k=1}^{n}[c_{\pi_k(\vec{i'}),\pi_k(\vec{j'})}(\alpha^{-k}(\alpha^n(x)))]^{-1}g'\\
&=\Pi_{k=1}^{n}[c_{\pi_k(\vec{i'}),\pi_k(\vec{j'})}(\alpha^{n-k}(x)))]^{-1}g'\\
&=\Pi_{k=1}^n[c_{\pi_k(\sigma_n(\vec{i'})),\pi_k(\sigma_n(\vec{j'}))}(\alpha^{k-1}(x))]^{-1}g'\\
&=\;[c^{(n)}_{\sigma_n(\vec{i'}),\sigma_n(\vec{j'})}(x)]^{-1}g'.
\end{split}\]

Then using transition function identities, we obtain
\[\begin{split}
c^{(n)}_{\vec{j},\sigma_n(\vec{j'})}(x)\cdot\;h\cdot h'\;&=\;c^{(n)}_{\vec{j},\sigma_n(\vec{j'})}(x)\cdot c^{(n)}_{\vec{i},\vec{j}}(x)g\cdot [c^{(n)}_{\sigma_n(\vec{i'}),\sigma_n(\vec{j'})}(x)]^{-1}g'\\
&=\;c^{(n)}_{\vec{i},\sigma_n(\vec{j'})}(x)[c^{(n)}_{\sigma_n(\vec{i'}),\sigma_n(\vec{j'})}(x)]^{-1}\cdot g\cdot g'\\
&=\;c^{(n)}_{\vec{i},\sigma_n(\vec{i'})}(x)\cdot c^{(n)}_{\sigma_n(\vec{i'}),\sigma_n(\vec{j'})}(x)\cdot [c^{(n)}_{\sigma_n(\vec{i'}),\sigma_n(\vec{j'})}(x)]^{-1}\cdot g\cdot g'\\
&=\;c^{(n)}_{\vec{i},\sigma_n(\vec{i'})}(x)\cdot g\cdot g'.
\end{split}\]
Thus $(c^{(n)}_{\vec{j},\sigma_n(\vec{j'})}(x)\cdot h\cdot h', x, 1, 0)\sim (c^{(n)}_{\vec{i},\sigma_n(\vec{i'})}(x)\cdot g\cdot g', x, 1, 0),$
and our proposed product is well-defined on equivalence classes in this case as well.

For (d), suppose  $n<0$, $n'=-n>0$, and $\alpha^n(x)=x'.$
Again we observe that $x'=\alpha^n(x)$ implies $\alpha^{-n}(x')=x,$ and our reasoning in the proof of (c) shows that for positive integers $-n,$
$x=\alpha^{-n}(x')\in\; W_{\vec{i}}$ if and only if $x'=\alpha^n(x)\in U_{\sigma_{-n}(\vec{i})}.$
It follows that $x'=\alpha^n(x)\in U_{\vec{i'}}\cap U_{\sigma_{-n}(\vec{i})}.$ This implies that
\begin{equation}
(g',x'=\alpha^n(x), \vec{i'}, -n)\sim (c^{(-n)}_{\vec{i'},\sigma_{-n}(\vec{i}))}(\alpha^n(x))g',\alpha^n(x)=x', \sigma_{-n}(\vec{i}),-n).
\end{equation}
The proof that this formula is consistent on equivalence classes is proved in a manner similar to that used in case (c).

For (e), suppose $n<0$, $n'>0$ and $|n|>n'$. Since $|n|=-n>n',$ write $-n=n'+m$ for $m>0,$ so that $n=-n'-m.$
Then Theorem~\ref{thm:twist_comp} (b) gives
\[(g,x,\vec{i},n)= (1_G,x, \vec{i_1},-m)\cdot (g, \alpha^{-m}(x), \vec{i_2}, -n'),\]
where $\vec{i}=(\vec{i_1},\vec{i_2})$ for $\vec{i_1}\in {\mathbb J}^m$ and $\vec{i_2}\in {\mathbb J}^{n'}$.
Then in order that the the groupoid product be associative, we must have
\[\begin{split}
&(g,x,\vec{i},n)\cdot (g',\alpha^n(x),\vec{i'},n')\\
&=\;(1_G, x, \vec{i_1},-m)\cdot \big((g, \alpha^{-m}(x), \vec{i_2}, -n')\cdot (g',\alpha^n(x),\vec{i'},n')\big).
\end{split}\]
By Theorem~\ref{thm:twist_comp} (d), the previous equation is equal to
\[\begin{split}
&\;(1_G, x, \vec{i_1},-m)\cdot (c^{(n')}_{\vec{i'},\sigma_{n'}(\vec{i_2})}(\alpha^{-n'}(\alpha^{-m}(x)))\cdot gg', \alpha^{-m}(x), 1,0)\\
&=\;(c^{(n')}_{\vec{i'},\sigma_{n'}(\vec{i_2})}(\alpha^{-n'-m}(x))\cdot g g', x, \vec{i_1},-m)\\
&=\;(c^{(n')}_{\vec{i'},\sigma_{n'}(\vec{i_2})}(\alpha^{n}(x))\cdot g g', x,\vec{i_1},n+n').
\end{split}\]
Thus for $n<0$, $n'>0$ with $|n|>n'$, we have
\begin{equation}\label{thm:(e)prod}
(g,x,\vec{i},n)\cdot (g',\alpha^n(x),\vec{i'},n')=((c^{(n')}_{\vec{i'},\sigma_{n'}(\vec{i_2})}(\alpha^{n}(x))\cdot g\cdot g', x,\vec{i_1},n+n'),
\end{equation}
where $(\vec{i_1},\vec{i_2})\;=\vec{i}$ for $\vec{i_1}\in {\mathbb J}^m$ and $\vec{i_2}\in {\mathbb J}^{n'}$.
We leave the proof that this formula for multiplication remains consistent on equivalence classes to the reader as it is similar to our earlier results.

For (f), suppose $n<0<n'$ with $|n|<n',$ we write $n'=-n+m$ for some integer $m>0,$ so that $m=n+n'$. Then Theorem~\ref{thm:twist_comp} (a) gives
\[(g',\alpha^n(x),\vec{i'},n')\;=\;(g', \alpha^n(x),\vec{i_1'}, -n)\cdot (1_G, x, \vec{i_2'}, m),\]
where $\vec{i_1'}\in {\mathbb J}^{-n},\;\vec{i_2'}\in {\mathbb J}^m,$ and $\vec{i'}=\;(\vec{i_1'}, \vec{i_1'})\in {\mathbb J}^{n'}.$
Because we require our groupoid product should be associative, and using Theorem~\ref{thm:twist_comp} (d) and \eqref{eq:flip_tf1}, we obtain
\[\begin{split}
(g,x,\vec{i},n)\cdot (g',\alpha^n(x),\vec{i'},n')
=\;(c^{(-n)}_{\vec{i_1'},\sigma_{-n}(\vec{i})}(\alpha^n(x))\cdot g\cdot g', x, \vec{i_2'}, n+n').
\end{split}\]
One checks that this formula for multiplication remains consistent on equivalence classes, just as above.

For (g), suppose $|n'|>n>0>n'$. Write $-n'=|n'|=n+m$ for some integer $m>0,$ so that $n+n'=-m<0.$
Note that $\vec{i'}\in {\mathbb J}^{-n'}={\mathbb J}^{n+m}.$ Thus Theorem~\ref{thm:twist_comp} (b) gives
\[(g',\alpha^n(x),\vec{i'},n')=(g', \alpha^n(x), \vec{i_1'}, -n)\cdot (1_G, x, \vec{i_2'}, -m),\]
where $\vec{i_1'}\in  {\mathbb J}^{n},\; \vec{i_2'}\in {\mathbb J}^{m},$ where $(\vec{i_1'},\vec{i_2'})=\vec{i'}\in {\mathbb J}^{-n'}.$

Again using the fact that the groupoid product must be assoiciative and Theorem~\ref{thm:twist_comp} (c), we have
\[\begin{split}
(g,x,\vec{i},n)\cdot (g',\alpha^n(x),\vec{i'},n')
=\;(c^{(n)}_{\vec{i},\sigma_n(\vec{i_1'})}(x)\cdot gg', x,  \vec{i_2'}, n+n').
\end{split}\]
Then one checks that this formula for multiplication remains consistent on equivalence classes as before.

For (h), suppose $n>|n'|>0>n',$ and write $n=|n'|+m=-n'+m$ for some integer $m>0$ so that $m=n+n'.$
Then Theorem~\ref{thm:twist_comp} (a) gives
\[\begin{split}
(g,x,\vec{i},n)&=(1_G, x, \vec{i_1}, m)\cdot (g, \alpha^m(x), \vec{i_2}, -n')=\;(1_G, x, \vec{i_1}, n+n')\cdot (g, \alpha^{n+n'}(x), \vec{i_2}, -n'),
\end{split}\]
where $\vec{i_1}\in {\mathbb J}^m$ and $\vec{i_2}\in {\mathbb J}^{-n'}$ are chosen so that $(\vec{i_1},\vec{i_2})=\;\vec{i}\;\in\;{\mathbb J}^n.$

The fact that our groupoid product must be associative once again together with Theorem~\ref{thm:twist_comp} (c) gives
\[\begin{split}
(g,x,\vec{i},n)\cdot (g',\alpha^n(x),\vec{i'},n')
=\;(c^{(-n')}_{\vec{i_2},\sigma_{-n'}(\vec{i'})}(\alpha^{n+n'}(x))\cdot gg', x, \vec{i_1}, n+n').
\end{split}\]
Then one checks that this formula for multiplication also remains consistent on equivalence classes.

Also it is straightforward to show that the inverse of $[(g,x,\vec{i},n)]\in\Lambda_n$ is given by $[(g^{-1},\alpha^n(x), \sigma_{|n|}(\vec{i}),-n)]$ using (c) and (d).

It remains to show that the product is associative and continuous.  We prove associativity in a typical case leaving the remaining cases to the reader, since as in our earlier discussion as regards consistency of multiplication formulas, it involves many cases.

Suppose that we have $\lambda_1=(g_1,x,\vec{i},m),\; \lambda_2=(g_2,\alpha^{m}(x),\vec{j},n)$ and $\lambda_3=(g_3,\alpha^{m+n}(x),\vec{k},p)$.
Let us suppose that $m<0,\; n>0,$ and $p<0$ with $|m|>n$ and $n<|p|.$
We first compute
$$
(\lambda_1\cdot \lambda_2)\cdot \lambda_3\;=\;([(g_1,x,\vec{i},m)]\cdot [(g_2,\alpha^{m}(x),\vec{j},n)])\cdot [(g_3,\alpha^{m+n}(x),\vec{k},p)].
$$
Since $m<0<n<|m|,$ write $\vec{i}=(\vec{i_1},\vec{i_2})$ with $\vec{i_1}\in {\mathbb J}^{-n-m}$ and $\vec{i_2}\in {\mathbb J}^{n}$, then  formula (e) gives:
\[
[(g_1,x,\vec{i},m)]\cdot [(g_2,\alpha^{m}(x),\vec{j},n)] =
[(c^{(n)}_{\vec{j},\sigma_{n}(\vec{i_2})}(\alpha^{m}(x))\cdot g_1\cdot g_2, x,\vec{i_1},m+n)].
\]
Note in this case $m+n<0$ so that we fall into case (b) and so
\begin{align*}
(\lambda_1\cdot \lambda_2)\cdot \lambda_3 &=
[(c^{(n)}_{\vec{j}, \sigma_{n}(\vec{i_2})}(\alpha^{m}(x))\cdot g_1\cdot g_2, x,\vec{i_1},m+n)]\cdot [(g_3,\alpha^{m+n}(x),\vec{k},p)] \\   &=
[(c^{(n)}_{\vec{j}, \sigma_{n}(\vec{i_2})}(\alpha^{m}(x))\cdot g_1g_2 g_3, x, (\vec{i_1},\vec{k}), m+n+p)].
\end{align*}
On the other hand, considering the case where $|p| > n>0>p,$ (which is case (g) of the Theorem), then, assuming
$(\vec{k_1},\vec{k_2})=\vec{k}\in {\mathbb J}^{-p},$ for $\vec{k_1}\in {\mathbb J}^{n}$ and $\vec{k_2}\in {\mathbb J}^{-p-n},$ we have
\begin{align*}
\lambda_1\cdot (\lambda_2 \cdot \lambda_3) &=
[(g_1,x,\vec{i},m)]\cdot ([(g_2,\alpha^{m}(x),\vec{j},n)]\cdot [(g_3,\alpha^{m+n}(x),\vec{k},p)]) \\
&=\;[(g_1,x,\vec{i},m)]\cdot ([c^{(n)}_{\vec{j},\sigma_n(\vec{k_1})}(\alpha^m(x))g_2g_3, \alpha^m(x), \vec{k_2}, n+p)]
\intertext{we have $m<0$ and $n+p<0$, so case (b) holds again and we obtain}
&=[(c^{(n)}_{\vec{j},\sigma_n(\vec{k_1})}(\alpha^m(x))\cdot g_1g_2g_3, x, (\vec{i},\vec{k_2}), m+n+p)].
\end{align*}
So to complete the proof of associativity, it suffices to show that
\begin{multline}\label{eq:associativity}
[(c^{(n)}_{\vec{j},\sigma_n(\vec{k_1})}(\alpha^m(x))\cdot g_1g_2g_3, x, (\vec{i},\vec{k_2}), m+n+p)]\\
=[(c^{(n)}_{\vec{j}, \sigma_{n}(\vec{i_2})}(\alpha^{m}(x))\cdot g_1g_2 g_3, x, (\vec{i_1},\vec{k}), m+n+p)]
\end{multline}
as elements of $\Lambda_{m+n+p}$.

Recall that for $n<0$, $(g,x,\vec{i},n) \in \Lambda_n$ is equivalent to
$(h,x,\vec{j},n)\in \Lambda_n$ if $x\in W_{\vec{i}}\cap W_{\vec{j}}$ and
$h=c^n_{\vec{i},\vec{j}}(x) g.$  Thus, showing \eqref{eq:associativity} amounts to showing that
\[
c^{(n)}_{\vec{j},\sigma_n(\vec{k_1})}(\alpha^{m}(x))c^{(m+n+p)}_{(\vec{i},\vec{k_2}),(\vec{i_1},\vec{k})}(x)
\;=\;c^{(n)}_{\vec{j},\sigma_{n}(\vec{i_2})}(\alpha^{m}(x)).
\]
Using Equation~\eqref{eq:negative_tf2} and the fact that both the first $-m-n$ components and last $-p-n$ components of the two vectors in question are the same, one calculates:
\[\begin{split}
&c^{(n)}_{\vec{j}, \sigma_n(\vec{k_1})}(\alpha^{m}(x))c^{(m+n+p)}_{(\vec{i},\vec{k_2}),(\vec{i_1},\vec{k})}(x)\\
&=\;c^{(n)}_{\vec{j}, \sigma_n(\vec{k_1})}(\alpha^{m}(x))\Pi_{l=1}^{n}[c^{(1)}_{\pi_l(\vec{i_2}),\pi_l(\vec{k_1})}(\alpha^{m+n-l}(x))]^{-1}\\
&=\;c^{(n)}_{\vec{j}, \sigma_n(\vec{k_1})}(\alpha^{m}(x))\Pi_{l=1}^{n}[c^{(1)}_{\pi_l(\vec{k_1}),\pi_l(\vec{i_2})}(\alpha^{m+n-l}(x))]\\
&=\;\Pi_{l=1}^n[c^{(1)}_{\pi_l(\vec{j}),\pi_{n+1-l}(\vec{k_1})}(\alpha^{m+l-1}(x))]\Pi_{l=1}^{n}[c^{(1)}_{\pi_l(\vec{k_1}),\pi_l(\vec{i_2})}(\alpha^{m+n-l}(x))]\\
&=\;\Pi_{l=1}^{n}[c^{(1)}_{\pi_l(\vec{j}),\pi_{n+1-l}(\vec{k_1})}(\alpha^{m+l-1}(x))]\Pi_{l=1}^{n} [c^{(1)}_{\pi_{n+1-l}(\vec{k_1}),\pi_{n+1-l}(\vec{i_2})}(\alpha^{m+l-1}(x))]\\
&=\;\Pi_{l=1}^{n}[c^{(1)}_{\pi_l(\vec{j}),\pi_{n+1-l}(\vec{k_1})}(\alpha^{m+l-1}(x))]\cdot [c^{(1)}_{\pi_{n+1-l}(\vec{k_1}),\pi_{n+1-l}(\vec{i_2})}(\alpha^{m+l-1}(x))]\\
&=\;\Pi_{l=1}^{n}[c^{(1)}_{\pi_l(\vec{j}),\pi_{n+1-l}(\vec{i_2})}(\alpha^{m+l-1}(x))]
=\;\Pi_{l=1}^{n}[c^{(1)}_{\pi_l(\vec{j}),\pi_{1}(\sigma_{n}(\vec{i_2}))}(\alpha^{l-1}(\alpha^m(x)))]\\
&=\;c^{(n)}_{\vec{j},\sigma_{n}(\vec{i_2})}(\alpha^{m}(x)).
\end{split}\]

Therefore the product is associative in this case, and in the other cases associativity is similarly shown.

As for continuity, we note that the maps $c^{(n)}_{\vec{i}, \vec{j}}$ are continuous on their open domains, and group multiplication is continous, and for each $n>0$ the map from $G\times  U_{\vec{i}}\times \{\vec{i}\} \times\{n\}$ to $\Lambda_n$ is continuous (being a quotient map), and for $n<0,$ the map from $G\times  W_{\vec{i}}\times \{\vec{i}\} \times\{n\}$ to $\Lambda_n$ is likewise continuous, and the map from $G\times X\times \{1\}\times\{0\}$ to $\Lambda_0$ is a homeomorphism.  It follows that multiplication as defined in the formulas above is continuous.
\end{proof}

\begin{rmk}\label{rmk:QHM2}
In the case of QHMs, the circle bundle $E$ over $\TT^2$ can be obtained from the symmetric $C(\TT^2)$-bimodule $M^c$ first constructed by Rieffel in \cite{Rie2}, and this will give $\Lambda_1,$ by \cite[Theorem~3.1]{DKM}.
 We only need two open sets $U_1$ and $U_2$ to cover $\TT^2$ to form $E$, and with transition functions $\{c_{i,j}:U_i\cap U_j\to \TT\}$, the circle bundle $E$ is given by $(\bigsqcup_{i=1}^2 \TT\times U_i\times \{i\})/\negthickspace\sim$, where $(g,x,i)\sim (g c_{ij}(x), x,j)$ for $x\in U_i$ and $g\in \TT$.
Then we obtain $\Lambda_n$ for each $n\in\ZZ$ as a circle bundle over $\TT^2$ using the refinement of the open covers $\{U_i\}_{i=1,2}$ described in the beginning of this section, and the twist groupoid $\Lambda$ is equal to $\bigsqcup_{n\in\ZZ} \Lambda_n$. The groupoid product of $\Lambda$ can be obtained using Theorem~\ref{thm:twist_comp}, and by the results of that Theorem it will be consistent on equivalence classes.
\end{rmk}

\end{document}